\documentclass[12pt]{article}
\usepackage{amsmath, amssymb, amsthm, amsfonts}
\usepackage{indentfirst}

\numberwithin{equation}{section}

\theoremstyle{plain}
\newtheorem{theorem}{Theorem} 
\newtheorem{lemma}{Lemma}\numberwithin{lemma}{section}
\newtheorem{proposition}{Proposition}\numberwithin{proposition}{section}
\newtheorem{corollary}{Corollary}\numberwithin{corollary}{section}
\theoremstyle{definition}
\numberwithin{definition}{section}
\theoremstyle{remark}
\newtheorem{remark}{Remark}\numberwithin{remark}{section}

\newcommand{\R}{\mathbb{R}}
\newcommand{\s}{\mathbb{S}}
\newcommand{\p}{\mathcal{P}}

\title {Absolute moments and Fourier-based probability metrics}

\author{Yong-Kum Cho}

\date{}

\begin{document}

\maketitle

\centerline{Department of Mathematics}
\centerline{College of Natural Science, Chung-Ang University}
\centerline{84 Heukseok-Ro, Dongjak-Gu, Seoul 156-756, Korea}

\medskip

\centerline{e-mail: ykcho@cau.ac.kr}

\bigskip

\begin{itemize}
\item[{}] {\bf Abstract.} We present a family of explicit formulae
for evaluating absolute moments of probability measures on $\mathbb{R}^d$
in terms of Fourier transforms. As to the space of probability measures
possessing finite absolute moments of an arbitrary order, we exploit
our formulae to characterize its Fourier image and construct
Fourier-based probability metrics which make the space complete. As applications, we compute
absolute moments of those probability measures whose characteristic functions belong to the
Scheonberg classes, estimate absolute moments of convolutions and investigate the
asymptotic behavior of solutions to the heat-diffusion equations from a probability view-point.
\end{itemize}

\bigskip
{\small
\begin{itemize}
\item[{}]{\bf Keywords.} Absolute moment, characteristic function, difference operator,
Fourier transform, heat-diffusion, probability metric, Scheonberg class.

\item[{}] 2010 Mathematics Subject Classification: 28A33, 42B10, 44A60, 60E10.
\end{itemize}}

\section{Introduction}
This paper is devoted to establishing an analysis framework based upon the Fourier
transform theory for studying probability measures defined on the Euclidean spaces which
offers computational algorithms for evaluating absolute moments and complete probability metrics.

For $\,d\ge 1,\,$ let $\R^d$ be the $d$-dimensional Euclidean space in which
the usual norm and inner product will be written as
$$|v| = \sqrt{v_1^2 + \cdots + v_d^2}\,,\quad v\cdot w= v_1 w_1 + \cdots + v_d w_d$$
for $\,v, w\in\R^d.\,$ Let $\mathcal{P}(\R^d)$ be
the set of all probability measures on $\R^d$, those nonnegative Borel measures with unit total mass.
For $\,\alpha>0,\,$ we denote by $\p_\alpha(\R^d)\,$ the set of all probability measures
whose absolute moments of order $\alpha$ are finite, that is,
\begin{equation}
\p_\alpha(\R^d) = \left\{\mu\in\p(\R^d): \int_{\R^d} |v|^\alpha d\mu(v)<\infty\right\}\,.
\end{equation}

\begin{itemize}
\item[(i)]
For $\,0<\alpha<\beta,\,$ $\,\p_\beta(\R^d)\subset\p_\alpha(\R^d)\,$
but the inclusion is proper. Indeed, given $\,\alpha>0,\,$
there exists $\,\mu\in \p_\alpha(\R^d)\,$ with $\,\mu\notin \p_\beta(\R^d)\,$ for any $\,\beta>\alpha.$
An example is given by
\begin{equation}
\mu_\alpha  =  \left(\sum_{k=1}^\infty\frac{1}{2^{k\alpha} k^2}\right)^{-1} \,
\sum_{k=1}^\infty\frac{1}{2^{k\alpha} k^2}\,\delta_{2^k e_1}\,,
\end{equation}
where $\,e_1=(1,0,\cdots, 0)\in\R^d\,$ and $\delta_{a}$ stands for
the Dirac measure supported at $a$ for any nonzero point $\,a\in\R^d.$
We remark that the probability measure $\mu_0$, which corresponds to the case $\,\alpha=0\,$
in this example, does not possess finite absolute moments of any order.
\item[(ii)] For a positive integer $m$, the moments of $\,\mu\in\p(\R^d)\,$ of order $m$ refers to the
unordered tuple of real values consisting of
$$\int_{\R^d} v^\sigma d\mu(v)\quad\text{for}\quad |\sigma|=m,$$
where $\,\sigma\in\mathbb{Z}_+^d\,$ and the usual multi-index notation is used. Clearly,
each $\,\mu\in\p_\alpha(\R^d)\,$ possesses finite moments up to order $\alpha$.
\end{itemize}

In dealing with probability measures, it is of great advantage to identify
them with functions by means of the Fourier transform which assigns to each
$\,\mu\in\p(\R^d)\,$ a unique continuous function $\widehat\mu$ defined by
\begin{equation}\label{1}
\widehat\mu (\xi) =  \int_{\R^d} e^{-i \xi\cdot v}\,d\mu(v)\quad(\xi\in\R^d).
\end{equation}

Let $\Phi(\R^d)$ denote the image of $\mathcal{P}(\R^d)$ under the Fourier
transform whose members are referred to as characteristic functions.
A well-known theorem of S. Bochner (\cite{Boc1}, \cite{Boc2}, 1932-3) states that $\,\phi\in\Phi(\R^d)\,$
if and only if $\,\phi(0) =1,\,$ it is continuous and positive definite, that is,
\begin{equation*}
\sum_{j=1}^n\sum_{k=1}^n\,\phi(\xi_j - \xi_k)\,z_j\overline{z_k} \ge 0
\end{equation*}
for any $\,\xi_1, \cdots, \xi_n\in\R^d,\, z_1, \cdots, z_n\in\mathbb{C}\,$ and any integer $n$.
Combined with the uniqueness property, Bochner's theorem asserts in fact that
the Fourier transform is a bijection from $\mathcal{P}(\R^d)$
onto $\Phi(\R^d)$.

An important consequence is that an analysis structure on $\Phi(\R^d)$ can be transferred into
the one on $\mathcal{P}(\R^d)$ via the Fourier transform.
For instance, it is evident by L\'evy's continuity theorem that the uniform metric induced by
$\,\|\phi\|_\infty = \sup_{\xi\in\R^d}\,\left|\phi(\xi)\right|\,$ makes the space $\Phi(\R^d)$ complete.
By transference, hence, $\mathcal{P}(\R^d)$ becomes a complete metric space with respect to
\begin{equation}
d_\infty(\mu, \nu) = \left\|\widehat \mu - \widehat \nu \right\|_\infty\,.
\end{equation}

As it is customary, any metric defined on a class of probability measures is
referred to as a probability metric. While $d_\infty$ is an example of a
complete probability metric on $\p(\R^d)$, its practical use is limited for it
is too insensitive to capture intrinsic properties of probability measures.
For instance, if $(\mu_t)_{t>0}$ denotes the family of Gaussian measures defined as
\begin{equation}
d\mu_t(v) = G_t(v) dv, \quad G_t(v) = (4\pi)^{-d/2} e^{-|v|^2/4t}\quad(t>0)
\end{equation}
and $\delta$ the Dirac measure supported at the origin, then it is simple to observe
$\,d_\infty(\mu_t, \delta) =1\,$ uniformly in $\,t>0.\,$ In view of the familiar fact
the Gaussian arises as the fundamental solution to the heat equation on $\R^d$, it may be expected
that $\,\mu_t \to \delta\,$ as $\,t\to 0\,$ and hence the metric $d_\infty$ is not
suitable in dealing with such a convergence matter.

Our concern is to construct complete Fourier-based probability metrics
on the $\p_\alpha(\R^d)$ in such a way that those
are sensitive enough to distinguish the absolute moments. As it is proven
to be useful in diverse situations, the problem of constructing such a metric
has become an important issue particularly in the theory of partial differential
equations.

Despite the fundamental role played by the Fourier transform in dealing with probability measures,
it is quite recently that the study on Fourier-based probability metrics had started.
It is initiated by E. Gabetta, G. Toscani and W. Wennberg (\cite{GTW}, 1995) who introduced the Fourier-based metrics
\begin{equation}\label{Fb1}
d_\alpha(\mu, \,\nu) =
\sup_{\xi\in\R^d}\frac{\,\left|\widehat{\mu}(\xi) - \widehat\nu(\xi)\right|\,}{|\xi|^\alpha}\quad(\alpha>0)
\end{equation}
in studying asymptotic behavior of solutions to the Boltzmann equation.
Subsequently, these metrics are shown to be effective in many other problems
(\cite{CK}, \cite{CGT}, \cite{Cho1}, \cite{GJT}, \cite{LT}, \cite{M}, \cite{PT}, \cite{TV}).
As it is shown by J. A. Carrillo and G. Toscani (\cite{CT}, 2007), each
$d_\alpha$ is well defined on the subclass of $\p_\alpha(\R^d)$ which consists of
probability measures having fixed moments up to order $[\alpha]$.

In their work on infinite energy solutions to the Boltzmann equation,
Y. Morimoto, T. Yang and S. Wang (\cite{MWY1}, 2015) introduced
another family of Fourier-based metrics defined by
\begin{equation}
d_{\alpha, \,\beta}(\mu, \nu) = d_\beta(\mu, \nu) +
\int_{\R^d} \frac{\,\left|\widehat{\mu}(\xi) - \widehat\nu(\xi)\right|\,}{|\xi|^{d+\alpha}}\,d\xi\,,
\end{equation}
where $\,0<\beta\le\min(\alpha, 1).\,$ In the range $\,0<\alpha<1,\,$ they characterized
the Fourier image of $\p_\alpha(\R^d)$ and proved that each $d_{\alpha, \,\beta}$ is a
complete probability metric on $\p_\alpha(\R^d)$. In addition, they
extended the results to $\,1\le\alpha<2\,$ for which $\p_\alpha(\R^d)$ is replaced by
\begin{equation}
\widetilde{\p_\alpha}(\R^d) = \left\{\mu\in\p_\alpha(\R^d) : \int_{\R^d} v^\sigma d\mu(v) = 0, \,\,|\sigma| = 1\right\}.
\end{equation}

In the present work we aim to obtain complete Fourier-based metrics on $\p_\alpha(\R^d)$
with any $\,\alpha>0\,$ and no additional assumptions. As it becomes clearer now how to proceed, it is essential
to characterize the Fourier image of $\p_\alpha(R^d)$ in the first place. For this purpose,
we start with deriving a list of formulae for evaluating absolute moments in terms of Fourier transforms.

The problem of evaluating absolute moments is of great importance
in the theory of probability and applied sciences due to its various applications.
Even for a probability metric, it is often inevitable to compute absolute moments in practice. To illustrate,
it is easy to see the distance between any $\,\mu\in\p_\alpha(\R^d)\,$
and the Dirac mass $\delta$ by the Wasserstein metric $W_\alpha$ is given by
\begin{equation}
W_\alpha(\mu, \delta) = \left(\int_{\R^d}|v|^\alpha d\mu(v)\right)^{1/\alpha}\quad(\alpha\ge 1)
\end{equation}
(see \cite{CT}, \cite{K}, \cite{V}, \cite{Wa} for the definition and properties).

An elementary calculation shows if $\,\mu\in\mathcal{P}_{2n}(\R^d)\,$ with $n$ an integer, then
\begin{equation}
\int_{\R^d} |v|^{2n} d\mu(v) = (-1)^n\,\sum_{|\sigma|=n} \frac{n!}{\sigma!}\,\left(\partial^{2\sigma}\widehat\mu\right)(0)
\end{equation}
and hence the problem of evaluating absolute moments in terms of Fourier transforms reduces to
the one for which the orders are not even integers.
While we are not aware of an explicit formula valid for arbitrary dimension $d$,
there are a number of formulae available when $\,d=1.\,$
Perhaps simplest is the formula of B. Brown (\cite{Brow}, 1972) which states

\begin{itemize}
\item[{}] {\bf Theorem.} {\it
Suppose $\alpha$ is not an even integer and $\,\alpha = k +\delta\,$ with $k$ an integer and $\,0<\delta\le 1.\,$
If $\,\mu\in\mathcal{P}_\alpha(\R)\,$ and $\,\widehat\mu =\phi,\,$ then
\begin{align}\label{Br}
\int_{-\infty}^{\infty} |v|^\alpha d\mu(v) &= C_\alpha
\int_0^\infty \Re\left[\phi^{(k)}(r) - \phi^{(k)}(0)\right] \frac{dr}{\,r^{1+\delta}},\\
\text{where}\,\,\,C_\alpha &=  2(-1)^{k+1}\sin\left(\alpha\pi/2\right)\Gamma(\delta +1)/\pi\,.\nonumber
\end{align}}
\end{itemize}

Here $\Re f $ stands for the real part if $f$ is a complex-valued
function. Brown's formula is based on the earlier version of P. Hsu (\cite{Hsu}, 1951) and B. von Bahr (\cite{Bahr1}, 1965).
We refer to M. Matsui and Z. Pawlas (\cite{MP}, 2014) for other developments and
surveyed references.

We evaluate and characterize absolute moments by means of the (finite forward) difference operator and its iterates.
Given $\,\xi\in\R^d,\,$ let $\Delta_\xi$ denote the difference operator with increment $\xi$ defined by
\begin{equation}
\Delta_\xi(\phi)(x) = \phi(x+\xi) - \phi(x)\quad(x\in\R^d)
\end{equation}
for each scalar function $\phi$ and $\,\Delta_\xi^k = \Delta_\xi\left(\Delta_\xi^{k-1}\right)\,$
its $k$th iterate defined by
\begin{equation}
\Delta_\xi^k(\phi)(x) = \sum_{m=0}^k\binom km (-1)^{k-m}\, \phi(x + m\xi)\quad(k\ge 1).
\end{equation}
If $\phi$ is the Fourier transform of a probability measure $\mu$ on $\R^d$, then
\begin{align}\label{DOI}
\Delta_\xi^k(\phi)(0) &= \sum_{m=0}^k\binom km (-1)^{k-m}\, \phi(m\xi)\nonumber\\
&=\int_{\R^d} \left( e^{-i \xi\cdot v} -1\right)^k\,d\mu(v).
\end{align}
In most of what follows, we shall deal with this type of iterated difference operators
in which the increment $\xi$ will be taken as a variable.

As a means of approximating derivatives, the difference operator and its iterates
arise frequently in the theory of function spaces. In particular, let us mention
some of our ideas owe to the work of R. Babby (\cite{Bag1}, \cite{Bag2}) in which the author constructed
certain functionals involving difference operators of type (\ref{DOI}) to characterize Riesz potentials and Sobolev spaces.

Our formulae for evaluating absolute moments read as follows.

\medskip

\begin{theorem}\label{theoremM1}
For $\,\alpha>0\,$ and a positive integer $k$, define
\begin{align}\label{M11}
&A(k, \alpha, d) =  -\, \frac{\,\sin \left(\frac{\alpha\pi}{2}\right)
\Gamma(\alpha +1)\Gamma\left(\frac{\alpha + d}{2}\right)}
{\pi^{\frac{d+1}{2}}\, S(k, \alpha)\Gamma\left(\frac{\alpha +1}{2}\right)}\,,\nonumber\\
&\text{where}\quad S(k, \alpha) = \sum_{m=1}^k \binom k m (-1)^{k-m}\,m^\alpha\,.
\end{align}
Let $\,\mu\in\mathcal{P}_\alpha(\R^d)\,$ and $\phi$ be the Fourier transform of $\mu$ defined as in (\ref{1}).

\begin{itemize}
\item[\rm{(i)}] If $\alpha$ is not an integer with $\,0<\alpha<k,\,$ then
\begin{equation}\label{M12}
\int_{\R^d}|v|^\alpha d\mu(v) = A(k, \alpha, d)
\int_{\R^d}\frac{\,\Delta_\xi^k(\phi)(0)\,}{\,|\xi|^{d+\alpha}}\,d\xi\,.
\end{equation}
\item[\rm{(ii)}] If $k$ is odd and $\alpha$ is not an integer with $\,0<\alpha<k+1\,$
or $\,\alpha =k,\,$ then
\begin{equation}\label{M13}
\int_{\R^d}|v|^\alpha d\mu(v) = A(k, \alpha, d)
\int_{\R^d}\frac{\,\Delta_\xi^k\left(\Re\,\phi\right)(0)}{\,|\xi|^{d+\alpha}}\,d\xi\,.
\end{equation}
Moreover, in the case when $\alpha$ is not an integer with $\,0<\alpha<k,\,$ both formulae
(\ref{M12}) and (\ref{M13}) coincide.
\end{itemize}
\end{theorem}

\medskip

\begin{remark} Our formulae cover any order except the even integer cases.
Due to continuity with respect to the order, however, the absolute moments of an even integer order
can be obtained as the limit
$$\int_{\R^d} |v|^{2n} d\mu(v) = \lim_{\alpha\,\to \,2n}\int_{\R^d} |v|^\alpha d\mu(v)\quad(n\in\mathbb{N})$$
provided $\,\mu\in\p_{2n +\epsilon}(\R^d)\,$ for some $\,\epsilon>0.$ In addition, we remark the following
aspects of our formulae for the sake of clarity.

\begin{itemize}
\item[(a)] The theorem remains valid without altering any material if $\mathcal{P}_\alpha(\R^d)$
is replaced by the space of all nonnegative Borel measures on $\R^d$ whose total masses and absolute moments of order $\alpha$
are finite.
\item[(b)] As it will be shown in Lemma \ref{lemmaS}, the sum $S(k, \alpha)$ is not zero in the stated ranges of
$k$ and $\alpha$. We also point out
$$\Delta_\xi^k\left(\Re\,\phi\right)(0) = \Re\left[\Delta_\xi^k(\phi)\right](0).$$
\item[(c)] A distinctive feature is that there are an infinite number of formulae available
for representing the absolute moment of a fixed order. For example, our formulae with $\,k=1, 2\,$ in the case $\,d=1\,$ yield
\begin{align*}
\int_{-\infty}^{\infty} |v|^\alpha d\mu(v)
&= -\frac{\sin\left(\frac{\alpha\pi}{2}\right)\Gamma(\alpha +1)}{\pi}
\int_{-\infty}^\infty \frac{\,\left(\Re\,\phi\right)(\xi) - 1\,}{|\xi|^{1+\alpha}}\,d\xi\\
& =  -\frac{\sin\left(\frac{\alpha\pi}{2}\right)\Gamma(\alpha +1)}{\pi (2^\alpha -1)}
\int_{-\infty}^\infty \frac{\,[\phi(2\xi) -2\phi(\xi) + 1]\,}{|\xi|^{1+\alpha}}\,d\xi
\end{align*}
for which the first one is valid for all $\,0<\alpha<2\,$ and the second one for all $\,0<\alpha<2, \,\alpha\ne 1.$
In comparison with Brown's formula, the first one coincides with (\ref{Br}) for $\,0<\alpha\le 1.\,$
\end{itemize}
\end{remark}

\medskip

By making use of our representation formulae for absolute moments, we shall characterize
the Fourier image of $\mathcal{P}_\alpha(\R^d)$ with an arbitrary order $\,\alpha>0.\,$
While we postpone its full statements until section 5, we summarize our main results as
follows, where $\mathcal{F}$ denotes the Fourier transform.

\medskip

\begin{theorem}\label{theorem2} Let $k$ be a positive integer.
\begin{itemize}
\item[\rm(i)] If $\alpha$ is not an integer with $\,0<\alpha<k,\,$ then
\begin{equation}\label{1.1}
\mathcal{F}\left(\p_\alpha(\R^d)\right) = \left\{\phi\in\Phi(\R^d) :
\int_{\R^d}\frac{\,\left|\Delta_\xi^k(\phi)(0)\right|}{\,|\xi|^{d+\alpha}}\,d\xi <\infty\right\}.
\end{equation}
\item[\rm(ii)] Assume that $k$ is odd and either $\alpha$ is not an integer with $\,0<\alpha< k+1\,$ or
$\,\alpha =k.\,$ Then
\begin{equation}\label{1.2}
\mathcal{F}\left(\p_\alpha(\R^d)\right) = \left\{\phi\in\Phi(\R^d) :
\int_{\R^d}\frac{\,\left| \Delta_\xi^k\left(\Re\,\phi\right)(0)\right|}{\,|\xi|^{d+\alpha}}\,d\xi <\infty\right\},
\end{equation}
which coincides with (\ref{1.1}) in the case $\,0<\alpha<k.$
\end{itemize}
\end{theorem}

\medskip

\begin{remark}
For a fixed $\,\alpha>0,\,$ a different choice of $k$ yields an equivalent characterization.
As it will be shown in section 6, it is also possible to
characterize the Fourier image in terms of functionals involving derivatives. In fact, if
$\,\alpha = m + \gamma\,$ with $m$ an integer and $\,0<\gamma<1\,,$ then
$\,\phi\in\mathcal{F}\left(\p_\alpha(\R^d)\right)\, $ if and only if $\,\phi\in\Phi(\R^d)\cap
C_b^m(\R^d)\,$ and
\begin{equation}
\int_{\R^d}\frac{\left|\left(\partial^\sigma \phi\right)(\xi) - \left(\partial^\sigma \phi\right)(0)\right|}{|\xi|^{d+\gamma}}\,d\xi
<\infty
\end{equation}
for every $\,\sigma\in\mathbb{Z}_+^d(\R^d)\,$ with $\,|\sigma|=m,\,$
where $C^m_b(\R^d)$ denotes the space of $C^m(\R^d)$-functions whose partial derivatives of order up to $m$ are bounded.
\end{remark}

\medskip

Based upon our characterizations of the Fourier images, we shall construct
families of complete probability metrics on $\p_\alpha(\R^d)$ in section 7. To
state some of our main results, we consider pseudo-metrics
\begin{align}
\left\|\phi - \psi\right\|_{\alpha, \,k} &=
\int_{\R^d}\,\frac{\left|\Delta^k_\xi(\phi-\psi)(0)\right|}{|\xi|^{d+\alpha}}\,d\xi\nonumber\\
&= \int_{\R^d}\,\biggl|\sum_{m=1}^k\binom km (-1)^{k-m}
\bigl[\phi(m\xi)-\psi(m\xi)\bigr]\biggr|\,\frac{d\xi}{\,|\xi|^{d+\alpha}}
\end{align}
for each pair of characteristic functions $\phi, \psi$. Let $d_\beta$ be the probability metric
defined as in (\ref{Fb1}) for each $\,\beta>0.\,$

\medskip

\begin{theorem}\label{theorem3}
Let $k$ be a positive integer.
\begin{itemize}
\item[\rm(i)] If $\alpha$ is not an integer with $\,0<\alpha<k\,$ and $\,0<\beta\le\min (\alpha, 1),\,$
then $\p_\alpha(\R^d)$ is complete with respect to the probability metric
\begin{equation}
F_{\alpha, \,\beta}(\mu, \nu) = d_\beta (\mu, \nu) + \left\|\,\widehat\mu - \widehat\nu\,\right\|_{\alpha, \,k}.
\end{equation}
\item[\rm(ii)] Assume that $k$ is odd and either $\alpha$ is not an integer with $\,0<\alpha< k+1\,$ or
$\,\alpha =k.\,$ Let $\,0<\beta\le\min (\alpha, 1).\,$ Then $\p_\alpha(\R^d)$ is complete with respect to the probability metric
\begin{equation}
H_{\alpha, \,\beta}(\mu, \nu) = d_\beta (\mu, \nu) + \left\|\,\Re\,\widehat\mu - \Re\,\widehat\nu\,\right\|_{\alpha, \,k}.
\end{equation}
\end{itemize}
\end{theorem}

\medskip

We shall present several applications of our results. With the aid of our formulae,
we evaluate explicitly absolute moments of probability measures whose characteristic functions
belong to the so-called Scheonberg classes which include
the stable L\'evy processes, Linnik distributions and Mittag-Leffler processes.
Our computations will be presented in section 8.

Next we apply our characterizations of Fourier images to estimate
the absolute moments of convolutions of probability measures in section 9.
In addition, we consider the heat-diffusion equations from a probabilistic
view-point in section 10. Applying our convolution theorem, we
estimate the moment propagations of solutions in time. We also exploit
Fourier-based probability metrics to obtain the asymptotic behavior
for large and small time, which improves the existing results with a new perspective.

We close this section with indicating additional references. For modern treatise
of Bochner's theorem, we refer to the book of H. Wendland
(\cite{Wend}, 2005). While we focus on Fourier-based ones at present,
there are a variety of probability metrics defined by other means and
we refer to the survey articles of A. Gibbs and F. Edward Su
(\cite{GES}, 2002) and V. M. Zolotarev (\cite{Z}, 1983) for their definitions,
applications and historical backgrounds, where the authors of the former
investigated relationships among popular probability metrics.

\section{Preliminaries}
\begin{lemma}\label{lemmaM}
For $\,-1<\delta<1,\,$
\begin{equation}\label{M1}
\int_0^\infty r^{-2-\delta}\sin^2 (r)\,dr =\frac{\,2^\delta\,\Gamma(1-\delta)\,}{\delta(1+\delta)}\,
\sin\left(\frac{\pi\delta}{2}\right)
\end{equation}
for which the case $\,\delta=0\,$ is interpreted as the limiting value $\,\pi/2.\,$
\end{lemma}

\smallskip

\begin{proof}
Making use of the Mellin transform (\cite{Erd}, p. 317)
\begin{align*}
\int_0^\infty r^{-1-\delta}\sin r\,dr &= 2^{1-\delta}\int_0^\infty r^{-1-\delta} \sin r\,\cos r\,dr\\
&= \frac{\,\Gamma(1-\delta)\,}{\delta}\,\sin\left(\frac{\pi\delta}{2}\right)\,,
\end{align*}
the formula (\ref{M1}) follows upon integrating by parts.
\end{proof}

\medskip

\begin{lemma}\label{lemmaS}
Let $k$ be a positive integer and put
\begin{equation*}
S(k, \alpha) = \sum_{m=1}^k \binom k m (-1)^{k-m}\,m^\alpha \quad(\alpha> 0),
\end{equation*}
the sum appeared in (\ref{M11}).
\begin{itemize}
\item[\rm{(i)}] $\,S(k, \alpha) =0 \,$ for $\,\alpha = 1, \cdots, k-1\,$ and $\,S(k, k) = k!.$
\item[\rm{(ii)}] If $\alpha$ is not an integer or $\,\alpha>k,\,$ there exists
$\,\theta\in (0, 1)\,$ such that
\begin{equation}\label{S3}
S(k, \alpha) = \alpha(\alpha-1)\cdots (\alpha -k+1)\,(\theta k)^{\alpha-k}\,.
\end{equation}
In particular, $\,S(k, \alpha) \ne 0.$
\end{itemize}
\end{lemma}

\smallskip

\begin{proof}
Our computation is based on a version of mean value theorems for finite differences:
Given
$\,r>0,\,$ if $\,f\in C^k((0, \infty))\,$ and continuous at $0$, then there exists
$\zeta$ between $0$ and $kr$ such that
\begin{equation}\label{S4}
\frac{\,\Delta_r^k(f)(0)\,}{r^k} = f^{(k)}(\zeta).
\end{equation}
Once this mean value theorem were verified, it is a simple matter to obtain the stated results by
an application with $\,f(x) = x^\alpha\,$ and $\,r=1.\,$

To prove (\ref{S4}), we consider the Lagrange polynomial $P$ interpolating $f$ at the $(k+1)$ points
$\,0, r, \cdots, kr,\,$ that is,
$$P(x)=\sum_{m=0}^k\,f(a_m)\,\biggl[\,\,
\prod_{\substack{0\le j\le k\\ j\neq m}}\,\frac{x-a_j}{a_m -
a_j}\,\biggr]\quad\text{with}\quad a_j= jr\,.$$
Let $\,g(x) = f(x) - P(x)\,.$ By Rolle's theorem applied
successively to the functions $\,g, g', \cdots, g^{(k-1)}\,,$ there exists
$\,\zeta\in (0, kr)\,$ with
$\,g^{(k)}(\zeta) =0.\,$ An easy computation yields
\begin{align*}
P^{(k)}(x) &= k! \sum_{m=0}^k\,f(a_m)\,\biggl[\,\,\prod_{\substack{0\le j\le k\\
j\neq m}}\,\frac{1}{a_m - a_j}\,\biggr] = \frac{\Delta_r^k (f)(0)}{r^k}
\end{align*}
for any $x$ and the desired mean value property follows at once.
\end{proof}

\medskip

\begin{lemma}\label{lemmaOC}
\rm{(Order-continuity)} For a given $\,\alpha_0>0,\,$ let
$\,\mu\in\mathcal{P}_{\alpha_0}(\R^d)\,$ and consider the absolute moment function
\begin{equation}\label{OC1}
M_\mu(\alpha) = \int_{\R^d} |v|^\alpha d\mu(v)\quad(0<\alpha\le\alpha_0).
\end{equation}
Then $M_\mu$ is continuous on $(0, \alpha_0)$ and left-continuous at $\alpha_0.$
\end{lemma}

\smallskip

\begin{proof}
H\"older's inequality implies
\begin{equation}\label{OC2}
\left[M_\mu(\beta)\right]^{1/\beta}\,\le\,\left[M_\mu(\alpha)\right]^{1/\alpha}
\end{equation}
for $\,0<\beta<\alpha\le\alpha_0,\,$ which shows $\,\alpha\mapsto\,\left[M_\mu(\alpha)\right]^{1/\alpha}\,$
is continuous on $(0, \alpha_0)$ and left-continuous at $\alpha_0.$ The assertion follows trivially.
\end{proof}

\section{Evaluations of oscillatory integrals}

In our evaluation process below, we shall use the following identities for the gamma function
which are valid for all complex $z$ unless the evaluation variables coincide with negative integers or zero:
\begin{align}
\Gamma(z+1) &= z\,\Gamma(z)\,,\label{G1}\\
\Gamma(1-z)\,\Gamma(z) &= \pi/\sin (\pi z)\,,\label{G2}\\
\Gamma(z)\,\Gamma\left(z+\frac12\right) &= 2^{-2z+1}\,\sqrt{\pi}\,\Gamma(2z)\,.\label{G3}
\end{align}

\medskip

\begin{lemma}\label{lemmaI}
Let $k$ be a positive integer and put
\begin{equation}\label{I1}
I(k, \alpha) = \int_0^\infty r^{-1-\alpha} \left[ \left(e^{-ir} - 1\right)^k + \left(e^{ir} - 1\right)^k\right]\,dr.
\end{equation}
Suppose $\,0<\alpha<k+1\,$ when $k$ is odd and $\,0<\alpha<k\,$ when $k$ is even.
Then the integral converges absolutely with the following explicit values:
\begin{itemize}
\item[\rm{(i)}] $\,I(k, \alpha) = 0\,$ for $\,\alpha =1, \cdots, k-1.\,$
\item[\rm{(ii)}] If $\alpha$ is not an integer or $\,\alpha =k\,$ with an odd integer $k$, then
\begin{equation}\label{I3}
I(k, \alpha) = - \frac{\,\pi\,S(k, \alpha)\,}{\,\,\sin\left(\pi \alpha/2\right)\,\Gamma(\alpha +1)\,}
\end{equation}
where $S(k, \alpha)$ denotes the sum defined in (\ref{M11}) and Lemma \ref{lemmaS}.
\end{itemize}
\end{lemma}

\smallskip

\begin{proof} Put $\,E_k(r) = \left(e^{-ir} - 1\right)^k + \left(e^{ir} - 1\right)^k.\,$
Since
\begin{equation}\label{I5}
E_k(r) = \left\{\aligned &{(-1)^{\frac{k+1}{2}}\, 2^{k+1} \,\sin^k\left(\frac{r}{2}\right)\,\sin \left(\frac{kr}{2}\right)}
\quad \text{if}\,\,k\,\,\text{is odd,}\\
&{\,\,\,\,\,(-1)^{\frac k2}\, 2^{k+1} \,\sin^k\left(\frac{r}{2}\right)\,\cos \left(\frac{kr}{2}\right)}
\quad \text{if}\,\,k\,\,\text{is even,}
\endaligned\right.
\end{equation}
it is elementary to deduce
\begin{equation}\label{I5-1}
\left|E_k(r)\right|\le \left\{\aligned &{\min \left(2^{k +1},\, k r^{k+1}\right)}
\quad \text{if}\,\,k\,\,\text{is odd,}\\
&{\min \left(2^{k+1},\, 2r^k \right)}
\quad\quad\text{if}\,\,k\,\,\text{is even.}
\endaligned\right.
\end{equation}
Thus it is evident the integral of (\ref{I1}) converges absolutely.

To evaluate $I(k, \alpha)$, we begin with the case $\,0<\alpha<2.\,$
Writing
\begin{align}\label{I7}
E_k(r) &= \sum_{m=0}^k \binom km (-1)^{k-m} \,\left(e^{-i mr } + e^{imr} -2\right)\nonumber\\
&= -4 \sum_{m=1}^k \binom km (-1)^{k-m}\,\sin^2 \left(\frac{mr}{2}\right)
\end{align}
and changing variables, we have
\begin{align}\label{I8}
I(k, \alpha) &= -2^{2-\alpha}\, \sum_{m=1}^k \binom km (-1)^{k-m}\,m^\alpha\left(\int_0^\infty r^{-1 -\alpha}\sin^2 (r)\,dr\right)
\nonumber\\
&= -2^{2-\alpha} \,S(k, \alpha)\, \int_0^\infty r^{-1 -\alpha}\sin^2 (r)\,dr\,.
\end{align}

With the aid of Lemma \ref{lemmaS} and \ref{lemmaM}, we compute $I(k, \alpha)$ explicitly.
\begin{itemize}
\item [(a)] $\,I(k, 1) = 0\,$ for $\,k\ge 2\,$ and
$$I(1, 1) = -2 \int_0^\infty r^{-2}\sin^2(r)\,dr = -\pi\,.$$
\item[(b)] For $\,\alpha\ne 1,\,$
\begin{align*}
I(k, \alpha) &= \frac{\,- 2\,S(k, \alpha)\,\Gamma(2 -\alpha)\,}{\alpha(\alpha-1)}\,\sin\left(\frac{\pi(\alpha-1)}{2}\right)\\
&= \frac{\,-\pi\,S(k, \alpha)\,}{\,\sin (\pi\alpha/2)\,\Gamma(\alpha +1)}\,,
\end{align*}
where the latter follows from (\ref{G1})-(\ref{G3}) and trigonometric identities.
\end{itemize}
This proves validity of the stated formula when $\,0<\alpha<2.\,$ In particular,
the cases of $\,k=1, 2\,$ are completely proved.

Assuming now $\,k\ge 4\,$ if $k$ is even and $\,k\ge 3\,$ if $k$ is odd, we divide the range of $\alpha$
into intervals $\,[2\ell, 2\ell +2)\,$ for $\,\ell =1, \cdots, (k-2)/2\,$ if $k$ is even and
for $\,\ell = 1, \cdots, (k-1)/2\,$ if $k$ is odd. Let us take one of such intervals and consider the case
$\,2\ell\le\alpha< 2\ell +2.\,$

Integrating by parts successively, if we exploit the fact $E_k(r)$ has a zero of order
at least $2(\ell +1)$ at $\,r=0\,$ and
$\,E_k(r), E_k{'}(r), \cdots, E_k^{(2\ell-1)}(r)\,$ are bounded, it is not hard to deduce
\begin{equation}\label{I13}
I(k, \alpha) = \frac{1}{\alpha(\alpha-1)\cdots (\alpha-2\ell +1)}\,
\int_0^\infty r^{2\ell - 1-\alpha} E_{k}^{(2\ell)}(r)\,dr\,.
\end{equation}
By using the identity
\begin{equation*}
\sum_{m=1}^k \binom km (-1)^{k-m}\,m^{2\ell} = S(k, 2\ell) = 0,
\end{equation*}
we may write
\begin{align}\label{I15}
E_k^{(2\ell)}(r) &= (-1)^\ell\sum_{m=1}^k \binom km (-1)^{k-m} \,m^{2\ell}\,\left(e^{-i mr } + e^{imr} -2\right)\nonumber\\
&= 4 (-1)^{\ell+1}\sum_{m=1}^k \binom km (-1)^{k-m}\,m^{2\ell}\,\sin^2 \left(\frac{mr}{2}\right)\,.
\end{align}
Inserting this into (\ref{I13}) and changing variables, we find
\begin{equation}\label{I16}
I(k, \alpha) = \frac{(-1)^{\ell+1} 2^{2(\ell+1)-\alpha}\,S(k, \alpha)}{\alpha(\alpha-1)\cdots (\alpha-2\ell +1)}\,
\int_0^\infty r^{2\ell -1 -\alpha}\sin^2 (r)\,dr\,.
\end{equation}

Lemma \ref{lemmaS} and \ref{lemmaM} yield the following:
\begin{itemize}
\item[(a)] $\,I(k, 2\ell) = I(k, 2\ell +1) = 0\,$ except the case $\,2\ell +1 = k,\,$ that is, when $k$
is odd and the interval corresponds to the last one for which
$$I(k, k) = 2(-1)^{\frac{k+1}{2}} \int_0^\infty r^{-2}\sin^2(r)\,dr = (-1)^{\frac{k+1}{2}}\pi\,.$$
\item[(b)] For a non-integral $\,2\ell<\alpha<2(\ell +1),\,$
$$I(k, \alpha) = \frac{2(-1)^{\ell+1}\,S(k, \alpha)\,\Gamma(2\ell +2 -\alpha)}{\alpha(\alpha-1)\cdots (\alpha-2\ell -1)}\,\sin\left(\frac{\pi(\alpha-2\ell-1)}{2}\right)\,.$$
By using the identities
\begin{align*}
\frac{\Gamma(2\ell +2 -\alpha)}{\alpha(\alpha-1)\cdots (\alpha-2\ell -1)}
&= \frac{-\pi}{\sin(\pi\alpha)\,\Gamma(\alpha +1)}\,,\\
\sin\left(\frac{\pi(\alpha-2\ell-1)}{2}\right)
&=(-1)^{\ell +1} \,\cos(\pi\alpha/2)\,,
\end{align*}
we easily deduce
$$I(k, \alpha) =\frac{\,\,-\pi\,S(k, \alpha)\,\,}{\,\sin (\pi\alpha/2)\,\Gamma(\alpha +1)}\,.$$
\end{itemize}

This proves validity of the stated formula in the case $\,2\ell\le\alpha <2(\ell +1).\,$
As it is verified independently of $\ell$, the proof is now complete.
\end{proof}

\section{Proof of Theorem \ref{theoremM1}}
Concerning part (i), let us assume that $\,\mu\in\p_\alpha(R^d)\,$ with $\,0<\alpha<k.\,$
For the Fourier transform $\,\phi=\widehat\mu,$ the identity (\ref{DOI}) gives
\begin{equation}\label{Pr1}
\int_{\R^d} \frac{\,\Delta_\xi^k(\phi)(0)}{|\xi|^{d+\alpha}}\,d\xi
=\int_{\R^d}\left[\int_{\R^d}\left(e^{-i\xi\cdot v} -1 \right)^k d\mu(v)\right] \frac{d\xi}{|\xi|^{d+\alpha}}
\end{equation}

For each non-zero $\,v\in\R^d,$ we observe
\begin{equation*}
\int_{\R^d}\frac{\,\left|e^{-i\xi\cdot v} -1\right|^k}{|\xi|^{d+\alpha}}\,d\xi
=|v|^\alpha \int_{\R^d}\frac{\,\left|e^{-i\eta\cdot v/|v|} -1\right|^k}{|\eta|^{d+\alpha}}\,d\eta.
\end{equation*}
In light of the estimate $\,|e^{-i\xi\cdot u} -1|^k\le\min( 2^k, \,|\xi|^k)\,$ for any unit vector $u$,
the integral on the right side obviously converges for $\,0<\alpha<k.$ Due to rotation invariance,
it is also independent of $v$ and hence
\begin{equation}\label{Pr2}
\int_{\R^d}\frac{\,\left|e^{-i\xi\cdot v} -1\right|^k}{|\xi|^{d+\alpha}}\,d\xi
\le C |v|^\alpha
\end{equation}
with a constant $\,C= C(k, \alpha, d)>0.$ Consequently,
\begin{equation}\label{Pr3}
\int_{\R^d}\int_{\R^d}\frac{\,\left|e^{-i\xi\cdot v} -1\right|^k}{|\xi|^{d+\alpha}}\,d\xi d\mu(v)
\le C\int_{\R^d} |v|^\alpha d\mu(v),
\end{equation}
which is finite for $\,\mu\in\p_\alpha(\R^d).\,$ By Fubini's theorem,
we may interchange the order of integrations in (\ref{Pr1}) to get
\begin{equation}\label{Pr4}
\int_{\R^d} \frac{\,\Delta_\xi^k(\phi)(0)}{|\xi|^{d+\alpha}}\,d\xi
=\int_{\R^d}\left[\int_{\R^d}\frac{\left(e^{-i\xi\cdot v} -1 \right)^k}{|\xi|^{d+\alpha}} d\xi\right] d\mu(v).
\end{equation}

To evaluate the inner integral, we first deal with the case $\,d\ge 2.\,$ For a fixed non-zero $\,v\in\R^d,\,$
we use the polar coordinates to write
\begin{equation*}
\int_{\R^d}\frac{\,\left(e^{-i\xi\cdot v} -1 \right)^k}{|\xi|^{d+\alpha}}\,d\xi
= |v|^\alpha\int_0^\infty r^{-1-\alpha}
\left[\int_{\s^{d-1}}\left(e^{-i r\omega\cdot v/|v|} -1 \right)^k\,d\sigma\right] dr,
\end{equation*}
where $\s^{d-1}$ denotes the unit sphere in $\R^d$.
Parameterizing $\s^{d-1}$ by the angle $\,\theta\in [0, \pi]\,$ defined by
$\,\omega\cdot v = \cos\theta\, |v|\,$ for each $\,\omega\in\s^{d-1},\,$ we have
\begin{align*}
&\int_{\s^{d-1}}\left(e^{-i r\omega\cdot v/|v|} -1 \right)^k\,d\sigma\nonumber\\
&\qquad =  \left|\s^{d-2}\right| \int_0^\pi \left(e^{-i r\cos\theta} -1 \right)^k\sin^{d-2}(\theta)\,d\theta\nonumber\\
&\qquad = \left|\s^{d-2}\right|\int_0^1\left[\left(e^{-irt} -1 \right)^k + \left(e^{irt} -1 \right)^k\right](1-t^2)^{\frac{d-3}{2}}\,dt,
\end{align*}
where $\left|\s^{d-2}\right|$ denotes the area of $(d-2)$-dimensional unit sphere
with the obvious interpretation $\,\left|\s^{d-2}\right| = 2\,$ when $\,d=2.$
Interchanging the order of integrations, we are led to
\begin{align}\label{Pr8}
\int_{\R^d}\frac{\,\left(e^{-i\xi\cdot v} -1 \right)^k}{|\xi|^{d+\alpha}}\,d\xi
&=  |v|^\alpha\,I(k, \alpha)\left|\s^{d-2}\right|
\int_0^1 t^\alpha (1-t^2)^{\frac{d-3}{2}}dt\nonumber\\
&=  -\, \frac {\pi^{\frac{d+1}{2}}\, S(k, \alpha)\Gamma\left(\frac{\alpha +1}{2}\right)}
 {\,\sin \left(\frac{\alpha\pi}{2}\right)
\Gamma(\alpha +1)\Gamma\left(\frac{\alpha + d}{2}\right)}\cdot |v|^\alpha
\end{align}
with $I(k, \alpha)$ the constant defined in (\ref{I1}) of Lemma \ref{lemmaI}.

In the case $\,d=1,\,$ it is simple to see the evaluation formula (\ref{Pr8})
continues to be valid. Integrating both sides, we obtain
\begin{equation}\label{Pr9}
\int_{\R^d} \frac{\,\Delta_\xi^k(\phi)(0)}{|\xi|^{d+\alpha}}\,d\xi
= -\, \frac {\pi^{\frac{d+1}{2}}\, S(k, \alpha)\Gamma\left(\frac{\alpha +1}{2}\right)}
 {\,\sin \left(\frac{\alpha\pi}{2}\right)
\Gamma(\alpha +1)\Gamma\left(\frac{\alpha + d}{2}\right)}\,\int_{\R^d}|v|^\alpha d\mu(v).
\end{equation}
According to Lemma \ref{lemmaS}, the sum $S(k, \alpha)$ is non-zero when $\alpha$ is non-integral and hence
(\ref{Pr9}) is equivalent to (\ref{M12}) for the multiplicative constant of (\ref{Pr9})
is nothing but the reciprocal of $A(k, \alpha, d)$ defined in (\ref{M11}).

Our proof for part (i) is complete.

As to part (ii), assume that $k$ is an odd integer and $\,0<\alpha<k+1.\,$ For the Fourier
transform $\,\phi=\widehat\mu,\,$ we have
$$ \Delta_\xi^k\left(\Re\,\phi\right)(0) = \int_{\R^d} \frac{(e^{-i\xi\cdot v} -1)^k +
(e^{i\xi\cdot v} -1)^k}{2}\,d\mu(v).$$
Owing to the evident estimate
\begin{equation*}
\left|\left(e^{-i\xi\cdot u} -1 \right)^k + \left(e^{i\xi\cdot u} -1 \right)^k\right|
\le \min\left(2^{k+1}, \,k |\xi|^{k+1}\right)
\end{equation*}
valid for any unit vector $u$, it is easy to see that
\begin{equation}
\int_{\R^d}\int_{\R^d}\frac{\bigl|(e^{-i\xi\cdot v} -1)^k + (e^{i\xi\cdot v} -1)^k\bigr|}{|\xi|^{d+\alpha}}
\,d\xi d\mu(v) \le C\int_{\R^d} |v|^\alpha d\mu(v)
\end{equation}
and we may interchange the order of integrations to deduce
\begin{equation*}
\int_{\R^d} \frac{\,\Delta_\xi^k\left(\Re\,\phi\right)(0)}{|\xi|^{d+\alpha}}\,d\xi
=\frac 12 \int_{\R^d}\int_{\R^d}\frac{(e^{-i\xi\cdot v} -1)^k + (e^{i\xi\cdot v} -1)^k}
{|\xi|^{d+\alpha}} d\xi d\mu(v)
\end{equation*}
in analogy with the identity (\ref{Pr4}).

Proceeding exactly in the same manner as before, it is plain to derive the desired formula
(\ref{M13}) once we recall $\,S(k, k) = k!\,$ in the present case. \qed

\section{Fourier images of $\mathcal{P}_\alpha(\R^d)$}
Having expressed absolute moments in terms of Fourier transforms, we now deal with the problem of
describing the precise range of $\mathcal{P}_\alpha(\R^d)$ under the Fourier transform. Our main theorem reads as follows.

\medskip

\begin{theorem}\label{theoremM2} Let $k$ be a positive integer.
\begin{itemize}
\item[\rm{(i)}] For $\,0<\alpha<k,\,$ if $\,\mu\in\mathcal{P}_\alpha(\R^d)\,$ and $\,\phi =\widehat\mu,\,$
then $\phi$ is a continuous positive definite function on $\R^d$ satisfying $\,\phi(0) =1\,$ and
there exists a constant $\,C_1= C_1(k, \alpha, d)>0\,$ such that
\begin{equation}\label{M22}
\int_{\R^d}\frac{\,\left|\Delta_\xi^k(\phi)(0)\right|}{\,|\xi|^{d+\alpha}}\,d\xi\,
\le\,  C_1 \int_{\R^d}|v|^\alpha d\mu(v).
\end{equation}
\item[\rm{(ii)}] Assume that $\alpha$ is not an integer with $\,0<\alpha<k.$ If $\phi$ is a continuous positive definite function
on $\R^d$ satisfying $\,\phi(0) =1\,$ and
\begin{equation}\label{M23}
\int_{\R^d}\frac{\,\left|\Delta_\xi^k(\phi)(0)\right|}{\,|\xi|^{d+\alpha}}\,d\xi <\infty,
\end{equation}
then $\phi$ is the Fourier transform of some $\,\mu\in\mathcal{P}_\alpha(\R^d)\,$ and
there exists a constant $\,C_2= C_2(k, \alpha, d)>0\,$ such that
\begin{equation}\label{M24}
C_2\int_{\R^d}|v|^\alpha d\mu(v)\,\le\,\int_{\R^d}\frac{\,\left|\Delta_\xi^k(\phi)(0)\right|}{\,|\xi|^{d+\alpha}}\,d\xi.
\end{equation}
\item[\rm{(iii)}] Suppose that $k$ is odd. If we replace $\,\Delta_\xi^k(\phi)(0)\,$ by
$\,\Delta_\xi^k\left(\Re\,\phi\right)(0)\,$ in the above
statements, then part {\rm{(i)}} remains valid for $\,0<\alpha<k+1\,$ and part {\rm{(ii)}} for non-integral $\alpha$ with
$\,0<\alpha<k+1\,$ or $\,\alpha =k.$
\end{itemize}
\end{theorem}

\smallskip

\begin{proof}(i) In view of Bochner's theorem, it suffices to verify that $\phi$ satisfies (\ref{M22})
for which the estimate (\ref{Pr3}) gives
\begin{align*}
\int_{\R^d}\frac{\,\,\left|\Delta_\xi^k(\phi)(0)\right|\,}{\,|\xi|^{d+\alpha}}\,d\xi
&\le \int_{\R^d}\int_{\R^d} \frac{\,\left|e^{-i\xi\cdot v} -1 \right|^k}{|\xi|^{d+\alpha}}\,d\xi\,d\mu(v)\\
&\le C\int_{\R^d} |v|^\alpha d\mu(v)
\end{align*}
with a constant $\,C= C(k, \alpha, d)>0.$

(ii) Invoking Bochner's theorem again, we may identify $\,\phi = \widehat\mu\,$
for some probability measure $\mu$ on $\R^d$. To verify $\,\mu\in\p_\alpha(\R^d),\,$
we denote by $B(k, \alpha, d)$ the multiplicative constant of (\ref{Pr8})
so that
\begin{equation*}
\int_{\R^d}\frac{\,\left(e^{-i\xi\cdot v} -1 \right)^k}{|\xi|^{d+\alpha}}\,d\xi
= B(k, \alpha, d) \,|v|^\alpha.
\end{equation*}
Note that $\,B(k, \alpha, d)\ne 0\,$ for non-integral $\alpha$ with $\,0<\alpha<k.$

For $\,\rho>0,\,$ let us consider
\begin{equation}\label{M26}
K_\rho = \frac{1}{B(k, \alpha, d) } \int_{\R^d} \int_{\{|\xi|>\rho\}} \frac{\,\left(e^{-i\xi\cdot v} -1 \right)^k}{|\xi|^{d+\alpha}}\,d\xi\,d\mu(v).
\end{equation}
Denoting by $\chi_E$ the indicator function of $\,E\subset\R^d,\,$ we observe
\begin{equation}\label{M27}
\left|\frac{\,\left(e^{-i\xi\cdot v} -1 \right)^k}{|\xi|^{d+\alpha}}\right|\,\chi_{\{|\xi|>\rho\}} \le 2^k |\xi|^{-d-\alpha} \,\chi_{\{|\xi|>\rho\}}
\end{equation}
uniformly in $\,v\in\R^d.\,$ Since for $\,d\ge 2,\,$
\begin{equation*}
\int_{\R^d} \int_{\{|\xi|>\rho\}} |\xi|^{-d-\alpha}\,d\xi\,d\mu(v) = \left|\s^{d-1}\right|\rho^{-\alpha} <\infty
\end{equation*}
and the same holds for $\,d=1\,$ with $\left|\s^{d-1}\right|$ replaced by $2$, we may interchange the order
of integrations, justified by Fubini's theorem, to find
\begin{align}\label{M210}
K_\rho &= \frac{1}{B(k, \alpha, d) } \int_{\{|\xi|>\rho\}} \frac{1}{|\xi|^{d+\alpha}}\left[ \int_{\R^d}
\left(e^{-i\xi\cdot v} -1 \right)^k\,d\mu(v)\right] d\xi\nonumber\\
&= \frac{1}{B(k, \alpha, d) } \int_{\{|\xi|>\rho\}}\frac{\,\Delta_\xi^k(\phi)(0)}{|\xi|^{d+\alpha}}\,d\xi,
\end{align}
which implies
\begin{equation}\label{M211}
\left| K_\rho\right| \le \frac{1}{\left|B(k, \alpha, d)\right|} \int_{\R^d} \frac{\,\left|\Delta_\xi^k(\phi)(0)\right|}{|\xi|^{d+\alpha}}\,d\xi
\end{equation}
uniformly in $\,\rho>0.\,$

By changing variables and using the rotation invariance, we may write
\begin{align}
K_\rho &= \int_{\R^d} |v|^\alpha H_\rho(v)\,d\mu(v)\quad\text{with}\nonumber\\
H_\rho(v) &= \chi_{\{v\ne 0\}}\cdot\frac{1}{B(k, \alpha, d) } \int_{\{|\eta|>\rho |v|\}}
\frac{\,\left(e^{-i\eta\cdot u} -1 \right)^k}{|\eta|^{d+\alpha}}\,d\eta,
\end{align}
where $u$ is a fixed unit vector in $\R^d$. As readily verified, $H_\rho(v)$ is real-valued and continuous in $\rho$.
Since $\,H_\rho(v)\to 1\,$ as $\,\rho\to 0,\,$ $\,H_\rho(v)>0\,$ for all sufficiently small $\,\rho>0.$
Therefore we may apply Fatou's lemma to deduce
\begin{align}\label{M223}
\int_{\R^d} |v|^\alpha d\mu(v) &\le \liminf_{\rho\to 0}\int_{\R^d} |v|^\alpha H_\rho(v)\,d\mu(v)\nonumber\\
&\le \frac{1}{\left|B(k, \alpha, d)\right|} \int_{\R^d} \frac{\,\left|\Delta_\xi^k(\phi)(0)\right|}{|\xi|^{d+\alpha}}\,d\xi,
\end{align}
where the latter results from the uniform boundedness of (\ref{M211}).
This proves $\,\mu\in\mathcal{P}_\alpha(\R^d)\,$ and (\ref{M24}) with the
lower bound $\,C_2 = \left|B(k, \alpha, d) \right|.\,$

(iii) As in the proof of Theorem \ref{theoremM1}, since the proof follows with minor modifications
of the proofs of (i) and (ii), we shall omit the details.

Our proof for Theorem \ref{theoremM2} is now complete.
\end{proof}

\medskip

We recall $\Phi(\R^d)$ denotes the collection of all continuous positive definite functions $\phi$ on $\R^d$
with $\,\phi(0) =1.\,$ For $\,\alpha>0\,$ and an integer $k$, let
\begin{align}
\Phi^{\,\alpha}_{k}(\R^d) &= \left\{\phi\in\Phi(\R^d)\,:\,
\int_{\R^d}\frac{\,\left|\Delta_\xi^k(\phi)(0)\right|}{\,|\xi|^{d+\alpha}}\,d\xi <\infty\right\},\label{FI1}\\
\Omega^{\,\alpha}_{k}(\R^d) &= \left\{\phi\in\Phi(\R^d)\,:\,
\int_{\R^d}\frac{\,\left|\Delta_\xi^k\left(\Re\,\phi\right)(0)\right|}{\,|\xi|^{d+\alpha}}\,d\xi <\infty\right\}.\label{FI1-1}
\end{align}

In a functional setting, Theorem \ref{theoremM2} yields the following:

\medskip

\begin{corollary}\label{corollaryFI2} {\rm{(Theorem \ref{theorem2})}}
Suppose $\,\alpha>0\,$ and it is not an even integer.
\begin{itemize}
\item[{\rm(i)}] For any integer $k$, if $\alpha$ is non-integral with $\,0<\alpha<k,\,$ then
the Fourier transform is a bijection from $\mathcal{P}_\alpha(\R^d)$ onto
$\Phi^{\,\alpha}_{k}(\R^d)$.
\item[{\rm(ii)}] For an odd integer $k$, if $\alpha$ is non-integral with $\,0<\alpha<k+1\,$
or $\,\alpha =k,\,$ then the Fourier transform is a bijection from $\mathcal{P}_\alpha(\R^d)$ onto
$\Omega^{\,\alpha}_{k}(\R^d)$.
\end{itemize}
\end{corollary}

A probability measure $\mu$ on $\R^d$ is said to be symmetric about the origin if
it is invariant under $\,v\mapsto -v\,$ on $\R^d$ in the sense that
\begin{equation}
\int_{\R^d} f(-v)\, d\mu(v) = \int_{\R^d} f(v)\, d\mu(v)
\end{equation}
for every bounded Borel measurable function $f$ on $\R^d$. For example, if $\mu$ is a radially symmetric
Lebesgue measure, say, $\,d\mu(v) = F(|v|) dv\,$ for some nonnegative function $F$ with unit mass on $[0, \infty)$,
then it is symmetric about the origin. In general, it can be shown
that $\mu$ is symmetric about the origin if and only if $\widehat\mu$ is real-valued
(see e.g. the book of V. Bogachev \cite{Bog}, p. 200).

With the notion of symmetry about the origin, we consider the following subclasses
of $\,\mathcal{P}_\alpha(\R^d), \,\Phi^{\,\alpha}_k(\R^d)\,$ for each $\,\alpha>0\,$ and an integer $k$.
\begin{align}\label{SO}
\mathcal{O}_{\alpha}(\R^d) &=\biggl\{ \mu\in\mathcal{P}_\alpha(\R^d)\, :
\,\mu\,\,\,\text{is symmetric about the origin}\,\,\biggr\},\nonumber\\
\widetilde{\Phi^{\,\alpha}_{k}}(\R^d) &=\biggl\{ \phi\in\Phi^{\alpha}_{k}(\R^d)\, :
\,\phi\,\,\,\text{is real-valued}\,\,\biggr\}.
\end{align}

An application of Corollary \ref{corollaryFI2} yields

\medskip

\begin{corollary}\label{propositionFI} For an integer $k$, suppose
either (i) $\alpha$ is non-integral with $\,0<\alpha<k+1\,$ or $\,\alpha=k\,$ when $k$ is odd
or (ii) $\alpha$ is non-integral with $\,0<\alpha<k\,$ when $k$ is even.
Then the Fourier transform is a bijection from $\mathcal{O}_{\alpha}(\R^d)$ onto $\widetilde{\Phi^{\,\alpha}_{k}}(\R^d).$
\end{corollary}

\section{Characterization by derivatives}
As the difference operator arises as a means of approximating derivatives, it may be expected that the Fourier image spaces
could be characterized by other functionals which involve derivatives. In fact,
if we make use of the following version of mean value theorem for the difference operator
and its iterates, then it is not hard to establish such a characterization.

\medskip

\begin{lemma} \label{lemmaP1}
\rm{(\cite{Cho2}, Lemma 3.1)}
For $\,\eta,\, \xi\in\R^d,\,$ if $\,\phi\in C^k(\R^d),\,$ then
\begin{align}
\frac{\Delta^k_\xi (\phi)(\eta)}{k!}
= \sum_{|\sigma|=k}\frac{\xi^\sigma}{\sigma !}
\int_{Q_k}\bigl(\partial^\sigma\phi\bigr)\bigl[\eta + (\theta_1 + \cdots +
\theta_k)\xi\bigr] d\theta_1\cdots d\theta_k,
\end{align}
where $\,Q_k = [0, 1]^k,\,$ the $k$-dimensional unit cube.
\end{lemma}

\medskip
According to Corollary \ref{corollaryFI2}, if $\alpha$ is positive but not an integer, then the Fourier transform is a bijection from
$\mathcal{P}_\alpha(\R^d)$ onto $\Phi^{\,\alpha}_{1+[\alpha]}(\R^d)$. In the case $\,\alpha>1,\,$ this image space
can be characterized as follows.

\medskip

\begin{proposition}\label{propositionCD}
Suppose that $\,\alpha = m + \gamma,\,$ where $m$ is a positive integer and $\,0<\gamma<1.\,$ A necessary and sufficient
condition for $\,\phi\in\Phi^{\,\alpha}_{m+1}(\R^d)\,$ is that $\,\phi\in \Phi(\R^d)\cap C_b^m(\R^d)\,$ and
\begin{equation}\label{P1}
\int_{\R^d} \frac{\,\,\left|\Delta_\xi\left(\partial^\sigma\phi\right)(0)\right|\,\,}{|\xi|^{d+\gamma}}\,d\xi<\infty
\quad\text{for all}\quad |\sigma|=m.
\end{equation}
Moreover, the following equivalence holds:
\begin{equation}\label{P2}
\int_{\R^d}\frac{\,\left|\Delta^{m+1}_\xi(\phi)(0)\right|\,}{|\xi|^{d+\alpha}}\,d\xi
\,\approx\,  \sum_{|\sigma|=m}\int_{\R^d} \frac{\,\,\left| \Delta_\xi\left(\partial^\sigma\phi\right)(0)\right|\,\,}
{|\xi|^{d+\gamma}}\,d\xi.
\end{equation}
\end{proposition}

\begin{proof}
For necessity, suppose $\,\phi\in\Phi^{\,\alpha}_{m+1}(\R^d).\,$ By Corollary \ref{corollaryFI2}, there exists a unique
probability measure $\,\mu\in\mathcal{P}_\alpha(\R^d)\,$ such that $\,\phi=\widehat\mu\,$ and
\begin{equation}\label{P3}
\int_{\R^d}|v|^\alpha\,d\mu(v)\,\approx\,\int_{\R^d}\frac{\,\left|\Delta^{m+1}_\xi(\phi)(0)\right|\,}{|\xi|^{d+\alpha}}\,d\xi.
\end{equation}
In view of the monotone property of absolute moments,
we may differentiate under the integral sign repeatedly to deduce that $\phi$ possesses continuous partial derivatives of order up to $m$ with
\begin{equation}\label{P4}
\left(\partial^\sigma\phi\right)(\xi) = \int_{\R^d} e^{-i \xi\cdot v} (-iv)^\sigma\,d\mu(v)\quad\text{for all}\quad |\sigma|\le m
\end{equation}
and hence $\,\phi\in C_b^m(\R^d)\,$ for which the boundedness of derivatives are obvious.
For each $\sigma$ with $\,|\sigma|=m,\,$ (\ref{P3}) and (\ref{P4}) give
\begin{align}\label{P5}
\int_{\R^d}\frac{\,\left|\Delta_\xi\left(\partial^\sigma\phi\right)(0)\right|\,}{|\xi|^{d+\gamma}}\,d\xi
&\le \int_{\R^d}|v|^m\left(\int_{\R^d} \frac{\,\left|e^{-i \xi\cdot v} -1\right|\,}{|\xi|^{d+\gamma}}\,d\xi\right) d\mu(v)\nonumber\\
&\le C\,\int_{\R^d} |v|^{m+\gamma}\,d\mu(v)\nonumber \\
&\le C\,\int_{\R^d}\frac{\,\left|\Delta^{m+1}_\xi(\phi)(0)\right|\,}{|\xi|^{d+\alpha}}\,d\xi,
\end{align}
where $C$ denotes a positive constant depending only on $\,\alpha, d\,$ which may be different in each occurrence.
This proves the condition (\ref{P1}).

For sufficiency, assume that $\,\phi\in \Phi(\R^d)\cap C_b^m(\R^d)\,$ and
each $\partial^\sigma\phi$ of order $\,|\sigma|=m\,$ satisfies the condition (\ref{P1}).
An application of Lemma \ref{lemmaP1} gives
\begin{align*}
\Delta_\xi^{m+1}(\phi)(0) &= \Delta_\xi^m(\phi)(\xi) - \Delta_\xi^m(\phi)(0)\\
&= m!\sum_{|\sigma|=m} \frac{\xi^\sigma}{\sigma!}\,
\int_{Q_m}\left[\left(\partial^\sigma\phi\right)\left(\left(1+ \overline{\theta}\right)\xi\right) -
\left(\partial^\sigma\phi\right)\left(\overline{\theta}\xi\right)\right] d\theta\,,
\end{align*}
where we put $\,\overline{\theta} = \theta_1 + \cdots +\theta_m\,$ and $\,d\theta = d\theta_1\cdots d\theta_m\,$ for simplicity.
By the elementary inequality $\,|\xi^\sigma|\le |\xi|^{|\sigma|},\,$ this identity implies
$$\int_{\R^d}\frac{\,\left|\Delta^{m+1}_\xi(\phi)(0)\right|\,}{|\xi|^{d+\alpha}}\,d\xi$$
is bounded above by
\begin{align}\label{P6}
&\,\sum_{|\sigma|=m}\frac{m!}{\sigma!}\,\biggl\{\int_{Q_m}
\int_{\R^d} \frac{\,\,\left| \left(\partial^\sigma\phi\right)\left(\left(1+ \overline{\theta}\right)\xi\right) - \left(\partial^\sigma\phi\right)(0)\right|\,\,}{|\xi|^{d+\gamma}}\,d\xi d\theta\nonumber\\
&\qquad\qquad\quad\quad + \int_{Q_m}
\int_{\R^d} \frac{\,\,\left| \left(\partial^\sigma\phi\right)\left(\overline{\theta}\xi\right) - \left(\partial^\sigma\phi\right)(0)\right|\,\,}{|\xi|^{d+\gamma}}\,d\xi d\theta\biggr\}\nonumber\\
& = \sum_{|\sigma|=m}\frac{m!}{\sigma!}
\,\int_{Q_m}\left[ \left(1+\overline{\theta}\right)^\gamma +\overline\theta^{\,\gamma}\right] d\theta\cdot
\int_{\R^d} \frac{\,\,\left| \Delta_\xi\left(\partial^\sigma\phi\right)(0)\right|\,\,}{|\xi|^{d+\gamma}}\,d\xi
\nonumber\\
&\le (2m+1) \sum_{|\sigma|=m}\frac{m!}{\sigma!}\,
\int_{\R^d} \frac{\,\,\left| \Delta_\xi\left(\partial^\sigma\phi\right)(0)\right|\,\,}{|\xi|^{d+\gamma}}\,d\xi.
\end{align}
Thus we conclude $\,\phi\in\Phi^{\,\alpha}_{m+1}(\R^d)\,$ and
\begin{equation}\label{P7}
\int_{\R^d}\frac{\,\left|\Delta^{m+1}_\xi(\phi)(0)\right|\,}{|\xi|^{d+\alpha}}\,d\xi
\le\,(2m+1)d^m\, \sum_{|\sigma|=m}\int_{\R^d} \frac{\,\,\left| \Delta_\xi\left(\partial^\sigma\phi\right)(0)\right|\,\,}{|\xi|^{d+\gamma}}\,d\xi,
\end{equation}
which follows from (\ref{P6}) in view of $\,|\sigma|!\le d^{|\sigma|}\sigma!\,$
for each $\,\sigma\in\mathbb{Z}_+^d.\,$

Finally, (\ref{P5}) and (\ref{P7}) give the equivalence (\ref{P2}).
\end{proof}

\medskip

\begin{remark}
An immediate consequence of Proposition \ref{propositionCD} is that
\begin{equation}
\left[\Phi(\R^d)\cap C^{1+[\alpha]}_b(\R^d)\right]\,\subset\,\mathcal{F}\left(\mathcal{P}_\alpha(\R^d)\right)
\end{equation}
for every non-integral $\,\alpha>0,\,$ which can be seen by applying the usual mean value theorem in the
functional defined on the right side of (\ref{P2}).
\end{remark}

\section{Moment-related metrics}
In this section we investigate the problem of constructing probability metrics on $\mathcal{P}_\alpha(\R^d)$
which make the space complete.

As the Fourier transform is a bijection from $\mathcal{P}_\alpha(\R^d)$ onto
$\Phi^{\,\alpha}_{k}(\R^d)$ for each non-integral $\,0<\alpha<k\,$ with $k$ an integer,
any complete metric on $\Phi^{\,\alpha}_{k}(\R^d)$
will be transferred to the one on the space $\mathcal{P}_\alpha(\R^d)$ in an obvious way.

For $\,\alpha>0\,$ and an integer $k$, let us introduce
\begin{equation}\label{me1}
\|\phi\|_{\alpha,\, k} = \int_{\R^d}\frac{\left|\Delta^k_\xi(\phi)(0)\right|}{|\xi|^{d+\alpha}}\,d\xi.
\end{equation}
As the $k$th difference operator $\Delta^k_\xi$ annihilates any polynomial in $\xi$ of degree less than $k$,
the map $\,(\phi, \psi)\mapsto \|\phi-\psi\|_{\alpha, \,k}\,$ given by
\begin{equation*}\label{me2}
\|\phi-\psi\|_{\alpha,\, k} = \int_{\R^d}\biggl|\sum_{m=1}^k\binom km (-1)^{k-m}
\bigl[\phi(m\xi)-\psi(m\xi)\bigr]\biggr|\,\frac{d\xi}{\,|\xi|^{d+\alpha}}
\end{equation*}
defines only a pseudo-metric on the space
$$\Phi^{\,\alpha}_{k}(\R^d) =
\bigl\{\,\phi\in\Phi(\R^d)\,: \,\|\phi\|_{\alpha, \,k} <\infty\,\bigr\}.$$

\medskip

\begin{lemma}\label{lemmaME1} For $\,\alpha>0\,$ and an integer $k$, suppose $(\phi_n)$ is a sequence
of complex-valued functions on $\R^d$ such that
(i) $\,\|\phi_n\|_{\alpha,\, k}<\infty\,$ for each $n$, (ii) it converges pointwise to $\phi$ and
(iii) $\,\|\phi_n -\phi_m\|_{\alpha, \,k}\,\to\,0\,$ as $\,n, m\to\infty.\,$
Then $\,\|\phi\|_{\alpha, \,k}<\infty\,$ and
$\,\|\phi_n -\phi\|_{\alpha, \,k}\,\to\,0\,$ as $\,n\to\infty.\,$
\end{lemma}

\smallskip

\begin{proof}
The Cauchy condition (iii) together with (i) implies
\begin{equation}
\|\phi_n\|_{\alpha, \,k} = \int_{\R^d}\frac{\left|\Delta_\xi^k(\phi_n)(0)\right|}{|\xi|^{d+\alpha}}\,d\xi \,\le\,C
\end{equation}
for some constant $\,C>0\,$ uniformly in $n$. Since $\,\Delta_\xi^k(\phi_n)(0)\,\to\,\Delta^k_\xi(\phi)(0)\,$
for each fixed $\xi$, Fatou's lemma gives
\begin{equation}
\int_{\R^d}\frac{\left|\Delta^k_\xi(\phi)(0)\right|}{|\xi|^{d+\alpha}}\,d\xi
\le \liminf\,\int_{\R^d}\frac{\left|\Delta_\xi^k(\phi_n)(0)\right|}{|\xi|^{d+\alpha}}\,d\xi \le C,
\end{equation}
that is, $\,\|\phi\|_{\alpha, \,k}<\infty.\,$ To see the convergence $\,\|\phi_n -\phi\|_{\alpha,\,k}\,\to\,0,\,$
let $\,\epsilon>0\,$ and choose an integer $N$ so that $\,\|\phi_n -\phi_m\|_{\alpha,\,k} <\epsilon/2\,$ for all
$\,n, m\ge N.\,$ For each fixed $\,n\ge N,\,$ we apply Fatou's lemma once again to find
\begin{align}
\int_{\R^d}\frac{\left|\Delta_\xi^k(\phi_n -\phi)(0)\right|}{|\xi|^{d+\alpha}}\,d\xi
\,&\le\, \liminf_{m\to\infty}\, \int_{\R^d}\frac{\left|\Delta_\xi^k(\phi_n -\phi_m)(0)\right|}{|\xi|^{d+\alpha}}\,d\xi\,\nonumber\\
&\le\, \epsilon/ 2\,,
\end{align}
whence $\,\|\phi_n -\phi\|_{\alpha,\,k} <\epsilon\,$ for all $\,n\ge N.\,$
\end{proof}

\medskip

In what follows, we shall not indicate the dependence on the number of iterates $k$ in
defining our metrics for the sake of simplicity in our notation.

An immediate consequence is the following:

\medskip

\begin{proposition}\label{propositionPM1}
For $\,\alpha>0,\,$ $\Phi^{\,\alpha}_{k}(\R^d)$ is complete with respect to
\begin{equation}\label{metric1}
D_\alpha(\phi, \psi) = \|\phi-\psi\|_\infty + \|\phi-\psi\|_{\alpha, \,k},
\end{equation}
where $k$ is any positive integer.
\end{proposition}

\medskip

It is often advantageous to replace the maximum part of (\ref{metric1})
by other moment-related metrics which still make $\Phi^{\,\alpha}_{k}(\R^d)$
complete. We consider the Fourier-based metric defined in (\ref{Fb1}) which may be recast as
\begin{equation}\label{me4}
d_\beta(\phi, \psi) = \sup_{\xi\in\R^d}\frac{\,\left|\phi(\xi) -\psi(\xi)\right|\,}{|\xi|^\beta}
\quad(\beta>0)
\end{equation}
for each pair $\,\phi, \psi\in\Phi(\R^d).\,$

An inspection reveals if $\,(\phi_n)\,$ is Cauchy
with respect to $d_\beta$, then it satisfies the pointwise Cauchy condition as well as
the uniform Cauchy condition on every compact subset of $\R^d$ with respect to the standard Euclidean metric and hence
$\,(\phi_n)\,$ converges pointwise to a continuous function.

Since $\,|\phi(\xi) -\psi(\xi)|
= \left|\Delta_\xi(\phi-\psi)(0)\right|\,$ whenever $\,\phi, \psi\in \Phi(\R^d),\,$
the natural space for which the metric $d_\beta$ makes sense is
\begin{equation}\label{me5}
\mathcal{K}^\beta(\R^d) = \left\{\,\phi\in\Phi(\R^d)\,:\,
\sup_{\xi\in\R^d}\frac{\,\left|\Delta_\xi(\phi)(0)\right|\,}{|\xi|^{\beta}} <\infty\,\right\},
\end{equation}
which is introduced by Cannone and Karch (\cite{CK}, 2010) in their study of infinite energy solutions to the Boltzmann
equation.

Our discussion leads almost immediately to the following.

\medskip

\begin{proposition}\label{propositionPM2}
For $\,0<\beta<2,\,$ $\mathcal{K}^\beta(\R^d)$ is complete with respect to $d_\beta$.
\end{proposition}

\smallskip

\begin{proof}
Let $\,(\phi_n)\subset\mathcal{K}^\beta(\R^d)\,$ be a Cauchy sequence with respect to $d_\beta$.
Then it converges to a continuous function by the above reasonings. If we denote by $\phi$
its limit, then $\,\phi\in\Phi(\R^d)\,$ by L\'evy's continuity theorem.

The Cauchy condition implies the uniform boundedness
$$\sup_{\xi\in\R^d}\frac{\,\left|\Delta_\xi(\phi_n)(0)\right|\,}{|\xi|^{\beta}}\,\le\,C,\quad n=1,2,\cdots,$$
which shows $\,\phi\in\mathcal{K}^\beta(\R^d)\,$ on account of $\,\Delta_\xi(\phi_n)(0)\,\to\,\Delta_\xi(\phi)(0).\,$
It is trivial to verify the convergence $\,d_\beta(\phi_n, \phi)\,\to\,0.\,$
\end{proof}

\medskip

\begin{remark} As it is pointed out in \cite{CK},
the reason for restricting $\,0<\beta<2\,$ is that the space $\mathcal{K}^\beta(\R^d)$ with $\,\beta\ge 2\,$
reduces to the singleton $\,\{1\}.\,$
\end{remark}

\medskip

\begin{lemma}\label{lemmaL1}
For $\,0<\beta\le k,\,$ if $\,\mu\in\mathcal{P}_\beta(\R^d)\,$ and $\,\phi=\widehat\mu,\,$ then
\begin{equation}\label{me6}
\sup_{\xi\in\R^d}\frac{\,\left|\Delta_\xi^k(\phi)(0)\right|\,}{|\xi|^{\beta}} \le 2^{k-\beta}
\int_{\R^d} |v|^\beta d\mu(v).
\end{equation}
\end{lemma}

\smallskip

\begin{proof}
By the obvious estimate
\begin{align*}
\sup_{\xi\in\R^d}\frac{\,\left|\Delta_\xi^k(\phi)(0)\right|\,}{|\xi|^{\beta}}
\le \int_{\R^d}\left(\sup_{\xi\ne 0}\frac{\,\left|e^{-i\xi\cdot v/|v|} - 1\right|^k\,}{|\xi|^\beta}
\right)\,\chi_{\{v\ne 0\}}\,|v|^\beta d\mu(v),
\end{align*}
the desired estimate follows upon noticing that the supremum part inside the parenthesis on the right side is bounded above by
\begin{align*}
2^{k-\beta}\,\sup_{r>0}\left\{\sup_{|t|\le 1}\,|t|^\beta\,\left|\sin\left(\frac{rt}{2}\right)\right|^{k-\beta}\,
\left|{\rm{sinc}}\left(\frac{rt}{2}\right)\right|^{\beta}\right\}
 \le 2^{k-\beta}
\end{align*}
with the notation $\,{\rm{sinc}}\,(x) = \sin x/x\,$ for $\,x\in\R.$
\end{proof}

\medskip

\begin{proposition}\label{propositionPM3} Suppose $\,\alpha>0\,$ and it is not an integer. Let $k$ be an integer
with $\,k>\alpha.$ Then
$\Phi^{\,\alpha}_k(\R^d)$ is complete with respect to the metric
\begin{equation}\label{metric2}
F_{\alpha, \,\beta}(\phi, \psi) = d_\beta(\phi, \psi) + \left\|\phi -\psi\right\|_{\alpha, \,k},
\end{equation}
where $\,0<\beta\le\min (\alpha, 1).\,$
\end{proposition}

\smallskip

\begin{proof}
In view of Lemma \ref{lemmaME1} and Proposition \ref{propositionPM2}, it suffices to show
$d_\beta$ is well defined on $\Phi^{\,\alpha}_k(\R^d)$. Let $\,\phi\in\Phi^{\,\alpha}_k(\R^d).\,$
By the second part of Theorem \ref{theoremM2}, $\,\phi=\widehat\mu\,$ for some $\,\mu\in\mathcal{P}_\alpha(\R^d)\,$
and there exists a constant $C_1$ with
$$\int_{\R^d}|v|^\alpha d\mu(v)\le C_1\,\|\phi\|_{\alpha, \,k}.$$
By the monotonicity of absolute moments and Lemma \ref{lemmaL1}, we deduce
\begin{align*}
\sup_{\xi\in\R^d}\frac{\,\left|\Delta_\xi(\phi)(0)\right|\,}{|\xi|^\beta}
\le 2^{1-\beta}\,\int_{\R^d}|v|^\beta d\mu(v)
\le C_2\,\|\phi\|_{\alpha, \,k}^{\beta/\alpha}
\end{align*}
for some constant $C_2$. Thus $\,\Phi^{\,\alpha}_{k}(\R^d) \subset \mathcal{K}^\beta(\R^d)\,$
and $\,d_\beta(\phi, \psi)<\infty\,$ for any pair of functions $\,\phi, \psi\in \Phi^{\,\alpha}_k(\R^d)\,$ as we wished to prove.
\end{proof}

\medskip
As for the space $\Omega^{\,\alpha}_k(\R^d)$, all of our reasonings can be modified easily to yield
the following:

\begin{proposition}\label{corollaryPM1} Let $k$ be an integer and $\,\alpha>0.$
\begin{itemize}
\item[\rm{(i)}] $\Omega^{\,\alpha}_k(\R^d)$ is complete with respect to the metric
\begin{equation}\label{metric3}
G_\alpha(\phi, \psi) = \|\phi -\psi\|_\infty + \left\|\Re\,\phi - \Re\,\psi\right\|_{\alpha, \,k}
\end{equation}
\item[\rm{(ii)}] If $k$ is odd and either $\alpha$ is non-integral with $\,0<\alpha<k+1\,$ or $\,\alpha =k,\,$
then $\Omega^{\,\alpha}_k(\R^d)$ is complete with respect to the metric
\begin{equation}\label{metric4}
H_{\alpha, \,\beta}(\phi, \psi) = d_\beta(\phi, \psi) + \left\|\Re\,\phi -\Re\,\psi\right\|_{\alpha, \,k},
\end{equation}
where $\,0<\alpha\le\min(\alpha, 1).\,$
\end{itemize}
\end{proposition}

\smallskip

Each metric structure described in Propositions \ref{propositionPM1}, \ref{propositionPM3} and
\ref{corollaryPM1} can be transferred into $\mathcal{P}_\alpha(\R^d)$ via the Fourier transform. To illustrate, let us consider
$D_\alpha$ which is defined in (\ref{metric1}) and a complete metric on the space $\Phi^{\,\alpha}_k(\R^d)$
for any integer $\,k>\alpha.\,$ For the sake of simplicity, we use the same notation $D_\alpha$ to define
$\,D_\alpha(\mu, \nu) = D_\alpha(\widehat\mu, \widehat\nu)\,$ for each pair
$\,\mu, \nu\in \p_\alpha(\R^d).\,$ Since the Fourier transform is a bijection from $\p_\alpha(\R^d)$ onto the space
$\Phi^{\,\alpha}_k(\R^d)$, it is evident that $D_\alpha$ is a complete probability metric on $\p_\alpha(\R^d)$.

By the same transference principle, if we extend each metric
$\,F_{\alpha, \,\beta}, G_\alpha, H_{\alpha, \,\beta}\,$ defined in (\ref{metric2}),
(\ref{metric3}), (\ref{metric4}), respectively, into $\p_\alpha(\R^d)$ by using
the Fourier transform and keep the same notation, we obtain the following.

\medskip

\begin{theorem} Let $k$ be an integer. For $\,\alpha>0,\,$ let $\,0<\beta\le\min(\alpha, 1).\,$
\begin{itemize}
\item[{\rm(i)}] If $\alpha$ is non-integral with $\,0<\alpha<k,$ then each of $\, D_\alpha, \,F_{\alpha, \,\beta},\,
G_\alpha,\,H_{\alpha, \,\beta}\,$ defines a complete probability metric on $\mathcal{P}_\alpha(\R^d).$
\item[{\rm(ii)}] If $k$ is odd and either $\alpha$ is non-integral with $\,0<\alpha<k+1\,$ or $\,\alpha=k,\,$
then each of $\,G_\alpha,\,H_{\alpha, \,\beta}\,$ defines a complete probability metric on $\p_\alpha(\R^d)$.
\end{itemize}
\end{theorem}

\section{Computations of absolute moments}
The purpose of this section is to illustrate the use of our formulae by evaluating
absolute moments of probability measures which arise frequently in the theory of probability and applied sciences.

\subsection{Stretched exponentials}
For each $\,p>0,\,$ we consider the function $\,\phi_p(\xi) = e^{-|\xi|^p},\,$ often referred to as the stretched
exponentials in applied sciences. It is shown by I. Scheonberg (\cite{S}, 1938) that $\phi_p$ is
positive definite on $\R^d$ for $\,0<p\le 2\,$ and not positive definite for $\,p>2.\,$
By Bochner's theorem,
there exists a unique probability measure $\mu_p$ on $\R^d$ satisfying $\,\phi_p =\widehat{\mu_p}\,$
for each $\,0<p\le 2.\,$

Evidently, $\phi_p$ is integrable on $\R^d$ and it follows from the Fourier inversion theorem
for measures that $\mu_p$ is indeed a probability density function on $\R^d$
in the sense $\, d\mu_p(v) = E_p(v) dv\,$ and
\begin{equation}\label{SE1}
E_p(v) = (2\pi)^{-d}\int_{\R^d} e^{i v\cdot\xi -|\xi|^p}\,d\xi\qquad(v\in\R^d, \,\,0<p\le 2)
\end{equation}
(see the book of M. Pinsky \cite{Pin}, p. 232, for instance).

For $\,a>-1/2,\,$ let $J_a$ be the Bessel function of the first kind of order $a$.
By a well-known formula for the Fourier transform of a radially symmetric function,
the Fourier inversion (\ref{SE1}) can be evaluated as
\begin{align}\label{SE2}
E_p(v) &=(2\pi)^{-d/2}\,|v|^{-(d-2)/2}\, \int_0^\infty e^{-r^p}r^{d/2} \,J_{\frac{d-2}{2}}(r|v|)\,dr\nonumber\\
&= \frac {2(4\pi)^{- d/2}}{p}\,\sum_{m=0}^\infty \frac{(-1)^m \,\Gamma\left(\frac{2m +d}{p}\right)}{m!\,\Gamma\left(\frac{2m +d}{2}\right)}
\,\left(\frac{|v|}{2}\right)^{2m}
\end{align}
valid for any dimension $d$ and $\,0<p\le 2.\,$ In the special case $\,p=1\,$ or $\,p=2,\,$
it is simple to represent this series expansion as a closed form
\begin{align}\label{SE3}
E_1(v) &=\Gamma\left(\frac{d+1}{2}\right)\,\left[\pi\bigl(1+|v|^2\bigr)\right]^{-(d+1)/2}\,,\nonumber\\
E_2(v) &= (4\pi)^{-d/2}\,e^{-|v|^2/4}\,.
\end{align}

The probability densities $E_p$ with $\,0<p\le 2\,$ are fundamental in many fields of mathematics
such as the probability theory of stable L\'evy processes and the theory of heat-diffusion equations (see the last section).
Concerning its absolute moments, R. Blumental and R. Getoor (\cite{BG}, 1960) proved
\begin{equation}\label{SE4}
\lim_{|v|\to\infty} |v|^{d+p} E_p(v) =\frac{p 2^{p-1}}{\pi^{d/2 +1}}\,\sin\left(\frac{p\pi}{2}\right)
\Gamma\left(\frac{d+p}{2}\right)\Gamma\left(\frac p2\right)
\end{equation}
and this asymptotic property implies that
each $E_p$ has finite absolute moments only of orders less than $p$ when $\,0<p<2.\,$
Of course, the Gaussian $E_2$ possesses finite absolute moments of all orders.

\medskip

\begin{proposition}\label{theoremE}
Let $\,0<p\le 2\,$ and $E_p$ be defined as in (\ref{SE1}) or (\ref{SE2}).
\begin{itemize}
\item[\rm{(E1)}] For $\,0\le\alpha<p<2\,,$
\begin{equation}\label{E1}
\int_{\R^d} E_p(v) |v|^\alpha dv
= \frac{\,2^\alpha\,\Gamma\left(1-\frac\alpha p\right)\Gamma\left(\frac{\alpha +d}{2}\right)\,}
{\Gamma\left(1-\frac\alpha 2\right)\Gamma\left(\frac{d}{2}\right)}\,.
\end{equation}
Moreover, the absolute moment of $E_p$ becomes singular as $\,\alpha\to p-\,\,$ with the asymptotic behavior \footnote{For real-valued functions
$\,A, B,\,$ the notation
$\,A(\alpha) \sim B(\alpha)\,$ as $\,\alpha\to\alpha_0\,$ means $$\,\lim_{\alpha\to\alpha_0}\, \frac{A(\alpha)}{B(\alpha)} =1\,.$$}
\begin{equation}\label{E2}
\int_{\R^d} E_p(v) |v|^\alpha dv  \,\sim \,\frac{\,2^p\,\Gamma\left(\frac{p +d}{2}\right)\,}{
\Gamma\left(1-\frac{p}{2}\right)\Gamma\left(\frac{d}{2}\right)}\cdot \left(1-\frac{\alpha}{p}\right)^{-1}\,.
\end{equation}
\item[\rm{(E2)}] For $\,p=2\,$ and $\,0\le\alpha<\infty\,,$
\begin{equation}\label{E3}
\int_{\R^d} E_2(v) |v|^\alpha dv  = \frac{\,2^\alpha\,\Gamma\left(\frac{\alpha +d}{2}\right)\,}{\Gamma\left(\frac{d}{2}\right)}\,.
\end{equation}
\end{itemize}
\end{proposition}

\medskip

Upon integrating in polar coordinates, it is easy to deduce
\begin{equation}\label{CC1}
\int_{\R^d}\frac{\Delta_\xi^k(f)(0)}{|\xi|^{d+\alpha}}\,d\xi = \left|\s^{d-1}\right|
\int_0^\infty r^{-1-\alpha}\left[\Delta_r^k(F)(0)\right] dr
\end{equation}
for a radially symmetric function $f$ on $\R^d$ with $\,f(\xi) = F(|\xi|),\,$ which will be
used as an alternative of (\ref{M12}) or (\ref{M13}) in the proof below.

\medskip

\begin{proof} (E1) Let $k$ be a positive integer. Write
$$\Delta_r^k\left(e^{-r^p}\right)(0) = \sum_{m=1}^k \binom km (-1)^{k-m}\left(e^{-(mr)^p} -1\right).$$
Integrating by parts and changing variables yield
\begin{align*}
\int_0^\infty r^{-1-\alpha}\left(e^{-(mr)^p} -1\right)dr = -\,\frac{m^\alpha\Gamma(1-\alpha/p)}{\alpha}
\end{align*}
for each $m$ and hence
\begin{align}\label{E4}
\int_0^\infty r^{-1-\alpha}\left[\Delta_r^k\left(e^{-r^p}\right)(0)\right] dr
= - S(k, \alpha)\frac{\Gamma(1-\alpha/p)}{\alpha}\,.
\end{align}
The formula (\ref{M13}), combined with (\ref{CC1}), gives (\ref{E1}) at once upon simplifying
the constant terms with the aid of (\ref{G1})-(\ref{G3}).

The asymptotic behavior of (\ref{E2}) follows from (\ref{E1}) upon noticing
$$\sin(\pi\alpha/p) = \sin\left(\pi(1-\alpha/p)\right) \,\sim\, \pi(1-\alpha/p)\quad\text{as}
\,\, \alpha\to p\,.$$

(E2) In the case $\,p=2,\,$ the absolute moment formula (\ref{E3}) follows easily from (\ref{SE3}). With no recourse to
this explicit density representation, it is also possible to obtain (\ref{E3}) directly
from the formula (\ref{M13}). In fact, by dividing the range of $\alpha$ into half-closed
intervals of length two and proceeding exactly in the same pattern as in the proof of Lemma \ref{lemmaI}, it is not hard to find
\begin{equation}
\int_0^\infty r^{-1-\alpha}\left[\Delta_r^k\left(e^{-r^2}\right)(0)\right] dr
= -\,\frac{\pi\,S(k, \alpha)}{\,2\sin(\pi\alpha/2)\Gamma(1+\alpha/2)\,}
\end{equation}
for $\,0<\alpha<k+1\,$ with an odd $k$ or $\,0<\alpha<k\,$ with an even $k$. Inserting it into
(\ref{CC1}), the formula (\ref{M13}) and simplifying the constant terms, we obtain (\ref{E3})
for each $\,\alpha>0\,$ unless $\alpha$ is not an even integer. By the order-continuity of Lemma \ref{lemmaOC},
(\ref{E3}) continues to be valid for all $\,\alpha\ge 0.\,$
\end{proof}

\subsection{Schoenberg classes}
For $\,0<p\le 2,\,$ the Schoenberg class $\Omega_p(\R^d)$ is defined to be the family of
function $\phi$ on $\R^d$ which admits the integral representation
\begin{equation}\label{Sh1}
\phi(\xi) = \int_0^\infty e^{-t|\xi|^p} d\nu(t)\quad(\xi\in\R^d)
\end{equation}
for some probability measure $\nu$ on $[0, \infty)$ (see L. Golinskii, M. Malamud and L. Oridoroga \cite{GMO}, 2015
and further references therein).

Closely related with the stretched exponentials or stable Lev\'y processes, it is evident
$\,\Omega_p(\R^d)\subset\Phi(\R^d)\,$ for $\,0<p\le 2.\,$ Interchanging the order of integrations,
Proposition \ref{theoremE} yields the following computations.

\medskip

\begin{proposition} Suppose $\,0<p\le 2\,$ and $\,\phi\in\Omega_p(\R^d)\,$ represented as (\ref{Sh1}).
Let $\mu$ be the probability measure determined by $\,\widehat\mu = \phi.$
\begin{itemize}
\item[\rm{(S1)}] For $\,0<\alpha<p<2\,,$
\begin{equation}\label{Sh3}
\int_{\R^d} |v|^\alpha d\mu(v)
= \frac{\,2^\alpha\,\Gamma\left(1-\frac\alpha p\right)\Gamma\left(\frac{\alpha +d}{2}\right)\,}
{\Gamma\left(1-\frac\alpha 2\right)\Gamma\left(\frac{d}{2}\right)}
\int_0^\infty t^{\alpha/p} d\nu(t)\,.
\end{equation}
\item[\rm{(S2)}] For $\,p=2\,$ and $\,\alpha\ge 0\,,$
\begin{equation}\label{Sh4}
\int_{\R^d} |v|^\alpha d\mu(v) = \frac{\,2^\alpha\,\Gamma\left(\frac{\alpha +d}{2}\right)\,}{\Gamma\left(\frac{d}{2}\right)}
\int_0^\infty t^{\alpha/p} d\nu(t)\,.
\end{equation}
\end{itemize}
\end{proposition}

As illustrations, we give a few examples which are of great importance in the theory of probability and interpolations.

\medskip

\noindent
{\bf (a) Multi-dimensional Linnik distributions.}  If we consider
$$ d\nu(t) = \frac{1}{\Gamma(\beta)}\,e^{-t} t^{\beta-1} dt\qquad(\beta>0), $$
the Gamma density probability measure, then it is simple to find
$$\int_0^\infty t^{\alpha/p} d\nu(t) = \frac{\,\Gamma\left(\beta + \frac{\alpha}{p}\right)\,}{\Gamma(\beta)}$$
and the formula (\ref{Sh3}), (\ref{Sh4}) give the absolute moments of the corresponding probability measures
whose characteristic functions are easily evaluated
\begin{equation}\label{Sh5}
\phi_p(\xi) = \int_0^\infty e^{-t|\xi|^p} d\nu(t)
= \left(1 +|\xi|^p\right)^{-\beta}\qquad(\xi\in\R^d).
\end{equation}

In the case $\,d=1,\,$ the distributions determined by $\phi_p$ are known as
the Linnik distributions. If $X_p$ denotes the random variable obeying the Linnik distribution
with the characteristic function $\phi_p$, then (\ref{Sh3}) simplifies to
\begin{equation}\label{Sh6}
E\left(\left|X_p\right|^\alpha\right) = \frac{\,2^\alpha\,\Gamma\left(1-\frac\alpha p\right)\Gamma\left(\frac{\alpha +1}{2}\right)\Gamma\left(\beta + \frac\alpha p\right)\,}
{\sqrt{\pi}\,\Gamma\left(1-\frac\alpha 2\right)\Gamma(\beta)}
\end{equation}
for all $\,0\le\alpha<p<2\,$ and (\ref{Sh4}) simplifies to
\begin{equation}\label{Sh7}
E\left(\left|X_2\right|^\alpha\right) = \frac{\,2^\alpha\,\Gamma\left(\frac{\alpha +1}{2}\right)\Gamma\left(\beta + \frac\alpha 2\right)\,}
{\sqrt{\pi}\,\Gamma(\beta)}
\end{equation}
for all $\,0\le \alpha<\infty\,$ ({\it cf.} G.-D. Lin \cite{Lin1}, 1998).

\medskip

\noindent
{\bf (b) Laplace transforms and Mittag-Leffler distributions.}  Suppose $F$ is
a probability distribution function on $[0, \infty)$ with its Laplace transform
$$\mathcal{L} (F)(s) = \int_0^\infty e^{-sr} dF(r)\qquad(s\ge 0).$$
As a special case of Schoenberg classes with $\,p=1\,$ and $\,d\nu(r) = dF(r),\,$
the radial extension of $\mathcal{L}(F)$ to $\R^d$
gives rise to a characteristic function on $\R^d$ for any dimension $d$.
If $\,\widehat\mu(\xi) = \mathcal{L}(F)(|\xi|)\,$ and $\,\widehat{E_1}(\xi) = e^{-|\xi|}\,$ as defined in (\ref{SE1}),
then our reasonings show
\begin{equation}\label{Sh8}
\int_0^\infty r^\alpha dF(r) = \dfrac{\,\int_{\R^d} |v|^\alpha d\mu(v)\,}{\,\int_{\R^d} E_1(v)|v|^\alpha dv\,}\quad (d\ge 1).
\end{equation}

As an application, we consider the Mittag-Leffler distributions
\begin{equation}
F_\delta(r) = \sum_{n=1}^\infty \frac{(-1)^n}{\Gamma(1+n\delta)}\,r^{n\delta}\qquad(r\ge 0,\,\,0<\delta\le 1).
\end{equation}
By W. Feller (\cite{Fe}, 1966), the Laplace transforms of $F_\delta$ are calculated as
$$\mathcal{L}\left(F_\delta\right)(s) = \left(1+ s^\delta\right)^{-1}\quad(s\ge 0).$$
The radial extensions $\,\widehat\mu_\delta(\xi) = (1+ |\xi|^\delta)^{-1}\,$
are nothing but the characteristic functions (\ref{Sh5}) with $\,p=1,\,$ which correspond to the Linnik distributions when $\,d=1.$
By (\ref{E2}), (\ref{Sh6}) and (\ref{Sh8}) with $\,d=1,\,$ we obtain
\begin{equation}
\int_0^\infty r^\alpha dF_\delta(r) =
\frac{\,\Gamma\left( 1-\frac\alpha\delta\right)\Gamma\left(1+\frac\alpha\delta\right)\,}{\Gamma(1-\alpha)}
\end{equation}
for $\,0\le \alpha<\delta\le 1.\,$

The associated stochastic process $\,\left(X_t\right)_{t\ge 0}\,$
is known as the Mittag-Leffler process for which $X_t$ has the distribution
\begin{equation}
F_{\delta, t}(r) = \sum_{n=0}^\infty
\frac{(-1)^n\,\Gamma(n+t)}{\,\,\Gamma(1+ (n+t)\delta)\Gamma(t)\,n!}\,r^{(n +t)\delta\,\,}\qquad(r\ge 0,\, t>0)
\end{equation}
and $\,X_0 = 0.\,$ As its Laplace transform is given by $\,(1 + s^\delta)^{-t}\,,$ we obtain
\begin{equation}
\int_0^\infty r^\alpha dF_{\delta, t}(r) =
\frac{\,\Gamma\left( 1-\frac\alpha\delta\right)\Gamma\left(t+\frac\alpha\delta\right)\,}{\Gamma(1-\alpha)\Gamma(t)}
\quad(t>0)
\end{equation}
for $\,0\le \alpha<\delta\le 1,\,$ which follows by the same reasonings
(see \cite{Lin2}, \cite{Pil} for other methods of computations).

\section{Absolute moments of convolutions}
We recall that the convolution of two complex Borel measures $\,\mu, \nu\,$ on $\R^d$ is the unique
complex Borel measure $\mu\ast\nu$ such that
$$\int_{\R^d} f(x)\, d(\mu\ast\nu)(x) = \iint_{\R^d\times\R^d} f(x+y)\, d\mu(x) d\nu(y)$$
for every bounded continuous $f$ on $\R^d$ (see e.g. W. Rudin \cite{R}).

Our aim here is to investigate the absolute moments for the convolution of two probability measures.
In terms of Fourier transforms, the convolution is given by
$\,(\mu\ast\nu)\,\widehat{}\,(\xi) = \widehat\mu(\xi)\,\widehat\nu(\xi).\,$
In order to apply Theorem \ref{theoremM2}, we shall need the well-known Leibniz rule
\begin{equation}\label{C2}
\Delta_\xi^k(\phi\psi)(0) = \sum_{m=0}^k\binom km \Delta_\xi^m(\phi)(0)\,\Delta_\xi^{k-m} (\psi)(m\xi),
\end{equation}
where $\Delta_\xi^0$ is interpreted as the identity operator.

In contrast to the familiar regularity-gaining property of convolution,
it turns out that the convolution may not increase the order of finite absolute moments.
To be precise, we have the following theorem where $M_\mu(\beta)$ denotes
the absolute moment of $\mu$ of order $\beta$ as defined in (\ref{OC1}).

\medskip

\begin{theorem}\label{theoremC1}
Suppose $\,\mu\in\mathcal{P}_\alpha(\R^d),\,\nu\in\mathcal{P}_\beta(\R^d)\,$ with $\,\alpha,\,\beta>0.\,$
Then the convolution $\,\mu\ast\nu\in \mathcal{P}_\gamma(\R^d)\,$ with
$\,\gamma = \min\,(\alpha, \,\beta)\,$ and there exists a constant $\,C= C(\alpha, \beta, d)>0\,$ such that
\begin{equation}\label{C1}
M_{\mu\ast\nu}(\gamma) \,\le\,C\,\left\{\bigl[M_\mu(\alpha)\bigr]^{\gamma/\alpha} + \bigl[M_\nu(\beta)\bigr]^{\gamma/\beta}\right\}.
\end{equation}
\end{theorem}

\medskip

\begin{remark} Since $\,\mu \ast\delta = \mu\,$ for any probability measure $\mu$
and $\delta$ has finite absolute moments of all orders,
Theorem \ref{theoremC1} can not be improved concerning the maximal order of finite absolute moments.
\end{remark}

\smallskip

\begin{proof} In view of the monotonicity (\ref{OC2}) of absolute moments, we note
\begin{equation}\label{CC2}
M_\mu(\gamma)\le \bigl[M_\mu(\alpha)\bigr]^{\gamma/\alpha}\quad\text{and}\quad M_\nu(\gamma) \le \bigl[M_\nu(\beta)\bigr]^{\gamma/\beta}.
\end{equation}

We first deal with the case when $\gamma$ is not an integer. Let
$\,\phi=\hat\mu,\,\psi=\hat\nu\,$ so that $\,(\mu\ast\nu)\,\widehat\,= \phi\psi\,$
and $\,k= 1+[\gamma].\,$ By the first part of Theorem \ref{theoremM2},
\begin{equation}\label{CC3}
\|\phi\|_{\gamma, \,k} \le C\, M_\mu(\gamma) \quad\text{and}\quad \|\psi\|_{\gamma, \,k} \le C\, M_\nu(\gamma)
\end{equation}
for some constant $\,C= C(\gamma, d)>0.\,$ Since  $\,\|\phi\|_\infty\le 1, \,\|\psi\|_\infty\le 1,\,$ it is plain to deduce from
the Leibniz rule of (\ref{C1}) the estimate
\begin{align}\label{C4}
\|\,\phi\psi\,\|_{\gamma, \,k} &\le \|\phi\|_{\gamma, \,k} + \|\psi\|_{\gamma, \,k} \nonumber\\
&+ \sum_{m=1}^{k-1}\binom km \int_{\R^d} \frac{\,\left|\Delta_\xi^m(\phi)(0)\right|\left|\Delta_\xi^{k-m}(\psi)(m\xi)\right|\,}{|\xi|^{d+\gamma}}\,d\xi\,.
\end{align}

Let us introduce the auxiliary parameters given as
$$\beta_m = \frac{m\gamma}{k},\quad \gamma_m = \frac{(k-m)\gamma}{k}\,,\quad m=1, \cdots, k-1.$$
Since $\,0<\beta_m<m,\, 0<\gamma_m<k-m,\,\beta_m +\gamma_m =\gamma,\,$ Theorem \ref{theoremM2} gives
\begin{align*}
\int_{\R^d} \frac{\left|\Delta_\xi^m(\phi)(0)\right|\left|\Delta_\xi^{k-m}(\psi)(m\xi)\right|}{|\xi|^{d+\gamma}}\,d\xi
&\le \|\phi\|_{\beta_m,\,m}\left[\sup_{\xi\in\R^d}\frac{\left|\Delta_\xi^{k-m} (\psi)(m\xi)\right|}{|\xi|^{\gamma_m}}\right]\\
&\le C\,2^{k-m}\,M_\mu(\beta_m)\, M_\nu(\gamma_m)\\
&\le C\,2^{k-m}\,\bigl[M_\mu(\gamma)\bigr]^{\frac{\beta_m}{\gamma}}\,\bigl[M_\nu(\gamma)\bigr]^{\frac{\gamma_m}{\gamma}},
\end{align*}
where $C$ is a constant depending only on $\gamma, d$ and the second inequality follows form
an obvious modification of the proof of Lemma \ref{lemmaL1}.

Combining with (\ref{CC2}), if $\,C= C(\gamma, d)\,$ denotes another constant which absorbs all of these estimates, the estimate (\ref{C1}) yields
\begin{align}
\|\,\phi\psi\,\|_{\gamma, \,k}
&\le C\, \sum_{m=0}^k\binom km \left[M_\mu(\gamma)\right]^{\frac{m}{k}} \left[M_\nu(\gamma)\right]^{\frac{k-m}{k}}\nonumber\\
&\le C\, \bigl\{ M_\mu(\gamma) + M_\nu(\gamma)\bigr\}.
\end{align}
By the second part of Theorem \ref{theoremM2}, we may conclude $\,\mu\ast\nu\in \mathcal{P}_\gamma(\R^d)\,$ and
the absolute moment inequality (\ref{C1}) follows from (\ref{CC2}) and (\ref{M24}).

Suppose now $\gamma$ is an integer. For each non-integral $\,0<\delta<\gamma,\,$
it follows from what verified in the above
$$ M_{\mu\ast\nu}(\delta) \le C\left\{[M_\mu(\gamma)]^{\delta/\gamma} +
[M_\nu(\gamma)]^{\delta/\gamma}\right\}\,.$$
In view of $\,
\left[M_{\mu\ast\nu}(\gamma)\right]^{1/\gamma} = \sup_{0<\delta<\gamma}\,\left[M_{\mu\ast\nu}(\delta)\right]^{1/\delta},
\,$ the result follows.
\end{proof}

\medskip

In the theory of probability and statistics, it is well-known that
if $\,X, Y\,$ are independent random variables with distributions $\,F, G,\,$ respectively,
then the sum $X+Y$ has the distribution $F\ast G$. In terms of random variables, hence,
Corollary \ref{corollaryC1} may be recast as the following for which
the case $\,d=1\,$ has been investigated by many authors (see e.g. \cite{Bahr2}, \cite{Ush}).

\medskip

\begin{corollary}\label{corollaryC1} For $\,\alpha>0\,$ and $\,\beta>0,\,$ put $\,\gamma = \min (\alpha, \,\beta).\,$
Suppose that $\,X, Y\,$ are independent random variables which take values in $\R^d$ and satisfy
$\,E\,|X|^\alpha <\infty\,$ and $\,E\,|Y|^\beta<\infty.\,$ Then $\,E\,|X+Y|^\gamma <\infty\,$ and
there exists a constant $\,C= C(\alpha, \beta, d)>0\,$ such that
\begin{equation}\label{C1}
E\,|X+Y|^\gamma \,\le\,C\,\left\{\bigl(E\,|X|^\alpha\bigr)^{\gamma/\alpha} + \bigl(E\,|Y|^\beta\bigr)^{\gamma/\beta}\right\}.
\end{equation}
\end{corollary}

\section{Heat-Diffusion equations}
For $\,0<p\le 2,\,$ we consider the heat-diffusion equation
\begin{equation}\label{H1}
\frac{\partial f }{\partial t}(x, t)  + \Lambda^p f (x, t) = 0\quad\text{for}\quad x\in\R^d,\,t>0,
\end{equation}
where $\Lambda^p$ denotes the differential operator defined
as $\, (\Lambda^p g)\,\widehat{}\,\,(\xi) = |\xi|^p \,\widehat g(\xi)\,$
for each integrable function $g$ on $\R^d$. From a probabilistic stand-point, it may be viewed
as the equation for the family of densities $(f(\cdot, t))_{t>0}$ associated with probability measures $\,(f_t)_{t>0}\,$
determined by $\,df_t(x) = f(x, t) dx.\,$

With an initial datum $\,f_0 =\mu\in\p(\R^d),\,$ we consider the Cauchy problem for (\ref{H1}).
In terms of characteristic functions, we reformulate
\begin{equation}
\left\{\aligned &{\frac{\partial \widehat f}{\partial t}(\xi, t)  + |\xi|^p\widehat{f} (\xi, t)
= 0\quad\text{for}\quad \xi\in\R^d, \,t>0,}\\
&\qquad\widehat{f}(\xi, 0) = \widehat\mu(\xi)\endaligned\right.
\end{equation}
for which an immediate solution is given by
\begin{equation}\label{H2}
\widehat{f}(\xi, t) = e^{-t|\xi|^p}\,\widehat\mu(\xi)\quad(\xi\in\R^d,\,t\ge 0).
\end{equation}
Let $E_p$ denote the density defined as in (\ref{SE1}). Upon considering
\begin{equation*}
E_p(x, t) = t^{-d/p} E_p\left(t^{-1/p} x\right)
\end{equation*}
for $\,t>0\,$, it is elementary to invert the Fourier expression of (\ref{H2}) as
\begin{equation}\label{H4}
f(x, t) = \int_{\R^d} E_p(x-y, t) d\mu(y)\quad (x\in\R^d, \,t>0).
\end{equation}

We remark that $E_p(x, t)$ is nothing but
the fundamental solution of (\ref{H1}), that is, the solution corresponding to the initial datum
$\,\mu = \delta.\,$ We also point out that $f$
is in fact the unique solution of (\ref{H1}) if we look for a solution, for example, in the space of continuous
functions $u$ on $\R^d\times(0, \infty)$ such that $\,\sup_{t>0} \|u(\cdot, t)\|_{L^1} <\infty\,$
(see Giga et. al. \cite{GGS}, 2010).

Our purpose here is to study the moment propagation and the asymptotic behavior of solutions
in time by applying our characterization theorems for the absolute moments and Fourier-based probability metrics.

\subsection{Moment propagation in time}
\medskip

\begin{proposition} For $\,0<p\le 2,\,$ let $f$ be the solution
(\ref{H4}) of the Cauchy problem (\ref{H1}) with $\,\mu\in\mathcal{P}_\alpha(\R^d)\,$ and
$\,df_t(x) = f(x, t) dx\,$ for each $\,t>0.$
\begin{itemize}
\item[\rm{(i)}] If $\,0<\alpha<p<2,\,$ then $\,f_t\in\mathcal{P}_\alpha(\R^d)\,$ for each $\,t>0\,$
and there exists a constant $\,C= C(\alpha, p, d)>0\,$ such that
\begin{equation}\label{H5}
\int_{\R^d} |x|^\alpha df_t(x) \le C (1 +t)^{\alpha/p}\int_{\R^d}|x|^\alpha d\mu(x).
\end{equation}
In the case when $\,p=2,\,$ the same holds for any $\,\alpha>0.$
\item[\rm{(ii)}] If $\,0<p<2\,$ and $\,\alpha\ge p,\,$ then $\,f_t\in\mathcal{P}_\beta(\R^d)\,$ for each
$\,0<\beta<p\,$ and $\,t>0\,$
and there exists a constant $\,C= C(\beta, p, d)>0\,$ such that
\begin{equation}\label{H5-2}
\int_{\R^d}|x|^\beta df_t(x) \le C (1 +t)^{\beta/p}\left(\int_{\R^d}|x|^\alpha d\mu(x)\right)^{\beta/\alpha}.
\end{equation}
\end{itemize}
\end{proposition}

\smallskip

\begin{proof}
Since $\,f(x, t) = \left[E_p(\cdot, t)\ast \mu\right](x)\,$ for $\,t>0,\,$ the result is an immediate consequence of
Theorem \ref{theoremC1} for the absolute moments of convolutions and Proposition \ref{theoremE} for the absolute moments of $E_p$
once we observe
\begin{align*}
\int_{\R^d} E_p(x, t) |x|^\alpha\,dx &= t^{\alpha/p}\int_{\R^d} E_p(x) |x|^\alpha\,dx\quad(t>0).
\end{align*}
\end{proof}

\subsection{Asymptotic behavior for large time}
We denote by $C_0^\infty(\R^d)$ the class of infinitely differentiable functions $g$ on $\R^d$ such that $g$ and all of its
partial derivatives vanish at infinity in the sense
$\,\left|\partial^\sigma g(x)\right|\to 0\,$ as $\,|x|\to\infty\,$ for every multi-index $\sigma$.

Let $f$ be the solution (\ref{H4}) of (\ref{H1}) with $\,\mu\in\mathcal{P}(\R^d).$
Since $\,|\xi|^m \widehat f(\xi, t)\,$ is integrable for any integer $m$, it follows from the Fourier inversion theorem and
the Riemann-Lebesgue lemma that
$\,f(\cdot, t)\in C^\infty_0(\R^d)\,$ and
\begin{align}\label{H6}
\left|\bigl(\partial^\sigma f\bigr)(x, t)\right|
&\le (2\pi)^{-d}\int_{\R^d} e^{-t|\xi|^p}|\xi|^{|\sigma|}\left|\widehat\mu(\xi)\right|\,d\xi\nonumber\\
& \le C_\sigma\, t^{-(d+|\sigma|)/p}\left\|\widehat\mu\right\|_\infty
\end{align}
uniformly in $x$ for every multi-index $\sigma$, where
$$C_\sigma = \frac{2\,\Gamma\left(\frac{d+|\sigma|}{p}\right)}{p\,(4\pi)^{d/2}\,\Gamma\left(\frac d2\right)}.
$$

Suppose $\,f, g\,$ are solutions of (\ref{H1}) with the initial data $\,\mu, \nu\in  \mathcal{P}(\R^d).\,$
By linearity, the estimate (\ref{H6}) yields
\begin{equation}\label{H7}
\bigl\|\,\partial^\sigma( f-g)(\cdot, t)\,\bigr\|_\infty
\le C_\sigma\,t^{-(d+|\sigma|)/p}\left\|\widehat\mu-\widehat\nu\right\|_\infty
\end{equation}
for every $\sigma$ and $\,t>0.\,$ From a functional analysis view-point, with the obvious topology on
$C^\infty_0(\R^d)$, it means the solution map $\,\mu\mapsto E_p(\cdot, t)\ast \mu\,$ is continuous from
$\mathcal{P}(\R^d)$ into $C^\infty_0(\R^d)$ and contractive for sufficiently large $t$.

Another interpretation is that any two solutions corresponding to different initial data
are essentially same as $\,t\to\infty\,$ in that the maximum distance between them or
between their derivatives tends to zero. In particular, if we consider the fundamental solution $E_p(x, t)$,
then it says the solution $f$ behaves asymptotically like $E_p(x, t)$ with
\begin{equation}\label{H8}
\bigl\|\,\partial^\sigma\left( f- E_p\right)(\cdot, t)\,\bigr\|_\infty = O\left(\,t^{-(d+|\sigma|)/p}\,\right)
\end{equation}
for every multi-index $\sigma$ and $\,t>0.\,$

If the initial data possess finite absolute moments, we may reformulate
the continuity (\ref{H7}) in terms of
the probability metric, for instance,
\begin{equation}
\rho_\alpha(\mu, \,\nu) = \int_{\R^d} \frac{\left|\widehat\mu(\xi) - \widehat\nu(\xi)\right|}{|\xi|^{d+\alpha}}\,d\xi
\end{equation}
and improve the convergence rate of (\ref{H8}) as follows.

\medskip

\begin{proposition} Suppose $\,\mu, \,\nu\in\mathcal{P}_\alpha(\R^d)\,$ with $\,0<\alpha<1.\,$
If $\,f, g\,$ are solutions of (\ref{H1}) with the initial data $\,\mu, \nu,$
respectively, then
\begin{equation}\label{H9}
\bigl\|\,\partial^\sigma( f-g)(\cdot, t)\,\bigr\|_\infty \le A_\sigma \, t^{-(\alpha + d+|\sigma|)/p}\,\rho_\alpha\left(\mu, \nu\right)
\end{equation}
for every multi-index $\sigma$ and $\,t>0,\,$ where $A_\sigma$ denotes the optimal bound
$$ A_\sigma = (2\pi)^{-d}\left(\frac{\alpha +d +|\sigma|}{ep}\right)^{\frac{\alpha +d +|\sigma|}{p}}.$$
\end{proposition}

\smallskip
\begin{proof}
For $\,x\in\R^d\,$ and $\,t>0,\,$ the Fourier inversion theorem gives
\begin{align*}
\bigl|\partial^\sigma (f-g)(x, t)\bigr| &= (2\pi)^{-d}\left|\int_{\R^d}
e^{i x\cdot\xi}(i\xi)^\sigma\left[\hat f(\xi, t) - \hat g(\xi, t)\right]\,d\xi\right|\\
&\le (2\pi)^{-d}\int_{\R^d}|\xi|^{|\sigma|}e^{-t|\xi|^p}\left|\widehat\mu(\xi)-\widehat\nu(\xi)\right|\,d\xi\\
&\le (2\pi)^{-d}\max_{\xi\in\R^d}\biggl(|\xi|^{\alpha + d+|\sigma|}e^{-t|\xi|^p}\biggr)
\rho_\alpha(\mu, \nu)
\end{align*}
and the result follows upon evaluating the maximum by calculus.
\end{proof}

\medskip
\begin{remark}
Our work on this continuity or stability property of solutions
are mainly motivated by that of T. Goudon, S. Junca and G. Toscani (\cite{GJT}, 2002)
where they obtained an analogous result for $f, g$ without considering derivatives 
in the $L^2$ setting by using
the probability metric $d_\alpha$ instead of $\rho_\alpha$.

In addition, we point out the following consequences.
\begin{itemize}
\item[(a)] Since it is evident
\begin{align*}
\rho_\alpha\left(\mu, \nu\right) \le C\left[ M_\mu(\alpha) + M_\nu(\alpha)\right],
\end{align*}
(\ref{H9}) gives an improvement of (\ref{H8}) in the form
\begin{equation}\label{H10}
\bigl\|\,\partial^\sigma\left( f- E_p\right)(\cdot, t)\,\bigr\|_\infty = O\biggl(\,M_\mu(\alpha)\,t^{-\frac{(\alpha + d+|\sigma|)}{p}}\,\biggr)
\end{equation}
for every multi-index $\sigma$ and $\,t>0.\,$
\item[(b)] We emphasize that the choice of the fundamental solution $E_p(x, t)$
is only for convenience. Alternatively, we may choose any other fixed solution
if it would be better suited in understanding the behavior of solutions
for large time. For example, in the case of the heat equation, that is, $\,p=2,\,$ we may take the centered-Gaussian
$$G_b(x, t) = (4\pi)^{-d/2} \,e^{-|x-b|^2/4t}$$
with any $\,b\in\R^d,\,$ the solution corresponding to $\delta_b$, for which
\begin{equation}\label{H11}
\bigl\|\,\partial^\sigma\left( f- G_b\right)(\cdot, t)\,\bigr\|_\infty =
O\biggl(\,\bigl[M_\mu(\alpha) +|b|^\alpha\bigr]\,t^{-\frac{(\alpha + d+|\sigma|)}{2}}\,\biggr)
\end{equation}
holds for every multi-index $\sigma$ and $\,t>0.\,$
\end{itemize}
\end{remark}

\subsection{Asymptotic behavior for small time}
\begin{proposition} For $\,0<p\le 2,\,$ let $\,0<\alpha<\min (1, p)\,$ when $\,0<p<2\,$ and
$\,0<\alpha<1\,$ when $\,p=2.\,$ Suppose that $f$ is the solution
(\ref{H4}) of the Cauchy problem (\ref{H1}) corresponding to the initial data $\,\mu\in\mathcal{P}_\alpha(\R^d)\,$ and
$\,df_t(x) = f(x, t) dx\,$ for each $\,t>0.$
Then
\begin{align}
\rho_\alpha\bigl(f_t, \,\mu\bigr) \le \left[\frac{2\pi^{d/2}\,\Gamma(1-\alpha/p)}{\alpha\,\Gamma(d/2)}\right] t^{\alpha/p}
\left\|\widehat\mu\right\|_\infty
\end{align}
for each $\,t\ge 0.\,$ As a consequence, $\,f_t \to \mu\,$ as $\,t\to 0+\,$ in the sense
\begin{equation}
\lim_{t\to \,0+}\,\rho_\alpha\bigl(f_t, \,\mu\bigr) = 0.
\end{equation}
\end{proposition}

\smallskip

\begin{proof}
By the definition of the metric $\rho_\alpha$, it is plain to observe that
\begin{align*}
\rho_\alpha\bigl(f_t, \,\mu\bigr) \le \left\|\widehat\mu\right\|_\infty \int_{\R^d}\frac{\left(1 - e^{-t|\xi|^p}\right)}{|\xi|^{d+\alpha}}\,d\xi
\end{align*}
and the result follows by direct calculus or by applying Proposition \ref{theoremE}.
\end{proof}

\vskip1cm

\noindent
{\bf{Acknowledgements.}} We are grateful to Professor Yoshinori Morimoto of Kyoto University
for his warm encouragement and keen interest on this work. In particular, we thank
for his informing us of the recent joint work \cite{MWY2} with S. Wang and T. Yang
on an extension of their previous results about the Fourier images and probability metrics.
Although their method and approach are different from ours, their work gave us many invaluable insights.

This research was supported by National Research Foundation of Korea Grant
funded by the Korean Government (\# 20150301).

\end{document}